\documentclass{article}%
\usepackage{amsmath}
\usepackage{amsfonts}
\usepackage{amssymb}
\usepackage{graphicx}
\usepackage{enumerate}
\usepackage{subcaption}
\usepackage{hyperref}
\usepackage{color}%
\providecommand{\U}[1]{\protect\rule{.1in}{.1in}}
\newtheorem{theorem}{Theorem}

\newtheorem{definition}[theorem]{Definition}

\newtheorem{lemma}[theorem]{Lemma}
\newtheorem{notation}[theorem]{Notation}

\newtheorem{proposition}[theorem]{Proposition}
\newtheorem{remark}[theorem]{Remark}

\newenvironment{proof}[1][Proof]{\noindent\textbf{#1.} }{\ \rule{0.5em}{0.5em}}

\numberwithin{equation}{section}
\numberwithin{theorem}{section}

\newcommand{\Img}{\operatorname{Im}}

\newcommand{\I}{\operatorname{i}}

\begin{document}

\title{The Heterogeneous {H}elmholtz Problem with Spherical Symmetry: Green's
Operator and Stability Estimates}
\author{Stefan Sauter\thanks{(stas@math.uzh.ch), Institut f\"{u}r Mathematik,
Universit\"{a}t Z\"{u}rich, Winterthurerstr 190, CH-8057 Z\"{u}rich,
Switzerland}
\and C\'eline Torres \thanks{(celine.torres@math.uzh.ch), Institut f\"{u}r
Mathematik, Universit\"{a}t Z\"{u}rich, Winterthurerstr 190, CH-8057
Z\"{u}rich, Switzerland}}
\maketitle

\begin{abstract}
We study wave propagation phenomena modelled in the frequency domain by the
Helmholtz equation in heterogeneous media with focus on media with
discontinuous, highly oscillating wave speed. We restrict to problems with
spherical symmetry and will derive explicit representations of the Green's
operator and stability estimates which are explicit in the frequency and the
wave speed.

\end{abstract}

\section{Introduction}

High-frequency scattering problems have many important applications which
include, e.g., radar and sonar detection, medical and seismic imaging as well
as applications in nano photonics and lasers. In physics, such problems are
studied intensively in the context of wave scattering in disordered media and
localisation of waves with the goal to design waves with prescribed intensity,
interference, localized foci, parity-time symmetry, etc.; see, e.g.,
\cite{Anderson1958}, \cite{LAGENDIJK1996143}, \cite{sebbah2001waves},
\cite{lahini2008anderson}, \cite{makris2017wave},
\cite{transmission_matrix_dis_sys}, \cite{lowe1995matrix} for references to
the theoretical and experimental physics literature.

In the frequency domain, these problems are often modelled by the Helmholtz
equation with, possibly, large wave number. For heterogeneous media, the
coefficients in these equation become variable and we focus here on the
effects of variable wave speed. In general, wave propagation in heterogeneous
media can exhibit interference phenomena in the form of, e.g., the
localisation of waves or energies which grow exponentially with respect to the
frequency. These events are rare in the set of all parameter configurations
-- however, their existence and their complicated behaviour make the analysis
notoriously hard.

In this paper, we discuss the Helmholtz problem on a bounded, Lipschitz domain
$\Omega\subset\mathbb{R}^{d}$%
\begin{equation}%
\begin{split}
-\Delta u-\left(  \frac{\omega}{c}\right)  ^{2}u  &  =f\quad\text{in }\Omega,\\
\mathcal{B}u  &  =g\quad\text{on }\Gamma:=\partial\Omega,
\end{split}
\label{eq:contHelmholtz}%
\end{equation}
for given right-hand side $f$ and boundary data $g$ and for a suitable first
order boundary differential operator $\mathcal{B}$ which guarantees existence
and uniqueness. We assume that the constant frequency satisfies%
\begin{equation}
\omega\geq\omega_{0}>0 \label{defomega} 
\end{equation}
and that the variable wave speed $c\in L^{\infty}(\Omega)$ is positive and
bounded: $0<c_{\min}\leq c\leq c_{\max}$ for some $c_{\min},c_{\max}>0$.

We define the \textquotedblleft energy space\textquotedblright{} by
$\mathcal{H}:=H^{1}(\Omega)$ equipped with the norm
\[
\Vert u\Vert_{\mathcal{H}}:=\sqrt{\left\Vert \nabla u\right\Vert
^{2}+\left\Vert \frac{\omega}{c}u\right\Vert ^{2}},
\]
where $\Vert\cdot\Vert$ denotes the standard $L^{2}$ norm on the domain
$\Omega$ induced by the $L^{2}$-scalar product denoted by $\left(  \cdot
,\cdot\right)  $ with the convention that complex conjugation is on the second
argument. Also we write $\left(  \cdot,\cdot\right)  _{\Gamma}:=\left(
\cdot,\cdot\right)  _{L^{2}(\Gamma)}$.

A stability estimate for the solution operator of the form (if available)%
\[
\Vert u\Vert_{\mathcal{H}}\leq C_{\operatorname{stab}}\left(  \Vert
f\Vert+\Vert g\Vert_{H^{1/2}(\Gamma)}\right)
\]
plays an important role in the design of numerical methods to approximate the
solution of \eqref{eq:contHelmholtz}. In \cite{GrahamSauter2017} the numerical
discretisation of heterogeneous Helmholtz equation is studied -- in
particular, a resolution condition is derived for abstract Galerkin
discretisations which explicitly depends on the stability constant
$C_{\text{stab}}$. Therefore the dependency of $C_{\text{stab}}$ on all the
parameters given in the problem (e.g. frequency $\omega$ and wave speed $c$)
is crucial for the design of a numerical discretisation.

In this work, we will present a fully explicit analysis of the stability of
the high frequency Helmholtz problem \eqref{eq:contHelmholtz} in a setting
with spherical symmetry.

\paragraph{Literature overview:}

First rigorous stability results for the heterogeneous Helmholtz equation go
back to Aziz et al. \cite{AKS1988}, where the $1$-dimensional problem with a
wave speed $c\in C^{1}(\Omega)$ is considered. For higher dimensions, similar
results are proved in \cite{PerthameVega1999} and \cite{GrahamSauter2017}, for
a wave speed that is assumed to be slowly varying. In \cite{Bellassoued2003}
the full-space problem with one $C^{\infty}$ inclusion and a discontinuity of
the wave speed $c$ across the interface is discussed. It is shown that the
stability constant in this setting cannot grow faster than exponential in the
frequency $\omega$. Transmission problems with one inclusion are also
considered in \cite{PopovVodev1999resonances} and \cite{CardosoPopovVodev1999}
(for a convex, $C^{\infty}$ obstacle), where conditions on $c$ are proposed
such the stability constant can grow super-algebraic in the frequency $\omega$
or such that the problem is stable independent of the frequency $\omega$. A
similar stability result is presented in \cite{MoiolaSpence_resonance_2017}
for a Lipschitz, star-shaped obstacle, together with a coefficient explicit
estimate on the stability constant. If the obstacle is a ball, an analysis of
the stability is carried out in \cite{Capdeboscq2011} (for $d=2$) and
\cite{CapLeadPark2012} (for $d=3$), using estimates of Bessel and Hankel
functions. The stability of the 1-dimensional Helmholtz problem with piecewise
constant and possibly highly oscillating wave speed $c$ is analysed in the
thesis \cite{ChaumontFreletDiss}. In the worst case, the estimate on the
stability constant grows exponentially in the number of discontinuities of
$c$. In \cite{SauterTorres2018} it is shown, for a one-dimensional model
problem, that for a wave speed $c$ which oscillates between two values, the
stability constant can grow at most exponentially in the frequency $\omega$
and is independent of the number of discontinuities of $c$.

Our approach follows the same basic idea as in \cite{SauterTorres2018}: since
the wave speed is assumed to be piecewise constant and spherically symmetric
we employ a Fourier expansion in the spherical variables and end up with a
radial transmission problem. We derive an explicit representation of the
Green's operator which is key for studying its stability.

\paragraph{Outline and main achievements of the paper:}

In this paper we investigate a heterogeneous Helmholtz problem in a spherical
symmetric setting in general dimension $d\in\mathbb{N}_{\geq1}$. In order to
focus on phenomena induced by the \textit{differential operator} we have
chosen $f=0$ as the right-hand side and inhomogeneous Dirichlet-to-Neumann
(DtN) boundary conditions. In Section \ref{subsec:fullspaceproblem} we
introduce the model problem and the class of parameters under consideration.

We consider piecewise constant wave speed jumping at $n$ radial points and the
emphasis is that the number of jumps $n$ may be arbitrary large. In Section
\ref{subsec:fullspaceproblem}, we employ a Fourier ansatz where the Fourier
coefficients then only depend on the radial variable and satisfy an ordinary
differential equation (ODE) of Bessel-type in each interval where the wave
speed is constant. Interface conditions are imposed at the jump points and
boundary conditions are derived for the ODE. The resulting system of equations
can be represented as a linear system of dimension $2n$, whose solution is
determined by the radial Green's operator (i.e. the inverse matrix of the system).

In the literature, often restrictions are imposed on the wave speed $c$ as,
e.g., $c$ belongs to $C^{0,1}\left(  \Omega,\mathbb{R}\right)  $; the wave
speed in radial coordinate satisfies a certain monotonicity behaviour; the
number of jumps equals $1$; the domain $\Omega$ is the full space and the
coefficient $c$ is periodic; the wave speed is given as a \textit{small}
fluctuation around the globally constant case; the Helmholtz equation is
considered in a stochastic setting; or the problem is restricted to the
one-dimensional case $d=1$. In contrast, the focus in this paper is on wave
numbers which do not satisfy such restrictions and to consider general
dimensions $d$.

Our first main result is the derivation of a new representation of the radial
Green's operator for general dimension $d$ and arbitrary number $n$ of jump
points. Since we restrict to a vanishing right-hand side $f=0$, only the last
column of the Green's operator is relevant. In Section \ref{sec:RepGreen} we
introduce the representation of the last column in the radial Green's operator%
\begin{equation}
\left\vert \left(  \mathbf{M}_{m}^{\operatorname{Green}}\right)  _{2\ell
,2n}\right\vert =\left\vert \frac{\operatorname{Im}\left(  \operatorname{e}%
^{\operatorname*{i}\frac{z_{\ell}}{c_{\ell+1}}}\beta_{m,\ell}\right)  }%
{\beta_{m,n}}\right\vert ,~\left\vert \left(  \mathbf{M}_{m}%
^{\operatorname{Green}}\right)  _{2\ell-1,2n}\right\vert =\left\vert
\frac{\beta_{m,\ell-1}}{\beta_{m,n}}\right\vert . \label{GreenFunRep}%
\end{equation}
Here, $m$ denotes the Fourier mode, $c_{\ell}$ are the constant values of the
wave speed and $z_{\ell}=\omega x_{\ell}$, where $x_{\ell}$ denotes the $\ell
$-th jump point in radial direction. The \textit{key} is the sequence $\left(
\beta_{m,\ell}\right)  _{\ell=1}^{n}\subset\mathbb{C}$ which satisfies a
simple linear recursion (see Remark \ref{RemLinRec}) and the analysis of the
Green's operator boils down to the investigation of this sequence. In
\cite{SauterTorres2018} for the one-dimensional case $d=1$, the Green's
operator has also been expressed by a recursive sequence in the complex plane.
However the recursion in \cite{SauterTorres2018} is more complicated via a
(rational) M\"{o}bius transform instead of the linear recursion in this paper.

The proof of representation \eqref{GreenFunRep} is technical and shifted to
Section \ref{sec:proofRep}. The representation suggests that the stability of
the original problem depends on the maximal growth/decay of $\left\vert
\beta_{m,\ell}\right\vert $ with respect to $\ell$ and we state in Section
\ref{subsec:mainstability} for the case $d=3$ and $m=0$ that the
maximum/minimum can grow/decay exponentially with respect to the frequency
$\omega$, in particular we prove
\[
|\beta_{0,\ell}|\leq\alpha^{\omega},\quad|\beta_{0,\ell}|\geq\alpha^{-\omega
}\quad\text{for some }\alpha>1
\]
which leads to a stability bound which is exponential with respect to $\omega
$. We present bounds for the energy norm $\left\Vert \cdot\right\Vert
_{\mathcal{H}}$ and also pointwise bounds. The proof of this main stability
result is postponed to Section \ref{sec:stabilitym0}.

In Section \ref{sec:examples}, we characterise different parameter configurations which lead either to a localisation effect in the solution or to a globally stable solution. In the first example, we fix $n=1$ and
recall from the literature how the choice of the frequency $\omega$ and the
Fourier mode $m$ may lead to a wave localisation along the single jump
interface of $c$ (also known as \textquotedblleft whispering gallery
modes\textquotedblright{}, see also \cite{CapLeadPark2012}). The other two
examples are new and show that the localisation effect can also occur for
$m=0$ if the number $n$ of jumps is \textquotedblleft in
resonance\textquotedblright\ with the frequency $\omega$. In both examples, we
consider the same wave speed $c$, they differ only on their choice of
frequency $\omega$. In the second example, we observe a localisation in the
centre of the domain, leading to an exponential growth with respect to
$\omega$ of the stability constant $C_{\operatorname{stab}}$ as $n\rightarrow
\infty$. This underlines the sharpness of our main stability result. In the
last example, although we also consider $n\rightarrow\infty$, the stability
constant $C_{\operatorname{stab}}$ stays bounded independently of $n$ and
$\omega$.

\section{Helmholtz problem with Spherical Symmetry\label{subsec:fullspaceproblem}}

In this section, we will specify the set of parameters (wave
speed/frequency/ boundary conditions) and introduce the spherical symmetric
setting which will be the basis of the Fourier expansion.

\subsection{The Helmholtz Problem with DtN Boundary
Conditions\label{Subsec:DtN}}

The Euclidean norm in $\mathbb{R}^{d}$ is denoted by $\left\vert
\cdot\right\vert $. We fix the domain $\Omega=B_{1}^{d}:=\{\mathbf{x}%
\in\mathbb{R}^{d}\mid\left\vert \mathbf{x}\right\vert <1\}$ and denote by
$\frac{\partial}{\partial n}$ the derivative in direction of the outward
normal vector. We set $f=0$ and the variable wave speed $c\in L^{\infty
}(\Omega)$ is assumed to be piecewise constant on annular regions. More
precisely, we assume that there are given points%
\begin{subequations}
\label{def:wavespeed}
 \begin{equation}
0=x_{0}<x_{1}<\ldots<x_{N}=1,  \label{defwavespeeda}%
\end{equation}
corresponding to intervals $\tau_{j}:=\left(  x_{j-1},x_{j}\right)  $ of
lengths $h_{j}:=x_{j}-x_{j-1}$ and positive numbers $c_{j}$ such that%
\begin{equation}
c\left(  \mathbf{x}\right)  =c_{j}\quad\quad\forall\mathbf{x}\in\Omega
\quad\text{with\quad}\left\vert \mathbf{x}\right\vert \in\tau_{j}\qquad
\forall1\leq j\leq N,  \label{defwavespeedb}%
\end{equation}%
\begin{equation}
0<c_{\min}\leq c\leq c_{\max}<\infty \label{defwavespeedc}%
\end{equation}
\end{subequations}
for some $c_{\min},c_{\max}\in\mathbb{R}_{>0}$. We consider the homogeneous
Helmholtz problem
\begin{equation}%
\begin{split}
-\Delta u-\left(  \frac{\omega}{c}\right)  ^{2}u  &  =0\quad\text{in }%
\Omega,\\
\frac{\partial u}{\partial n}-T_{\frac{\omega}{c_{N}}}u  &  =g\quad\text{on
}\Gamma
\end{split}
\label{modelproblemDtN}%
\end{equation}
with inhomogeneous Dirichlet-to-Neumann boundary conditions which are defined
as follows. Let $\Omega^{+}:=\mathbb{R}^{d}\backslash\overline{\Omega}$. It
can be shown that, for given $\tilde{g}\in H^{1/2}\left(  \Gamma\right)  $ and
$\kappa\in\mathbb{R}$, the problem:%

\[%
\begin{array}
[c]{l}%
\text{find }w\in H_{\operatorname*{loc}}^{1}\left(  \Omega^{+}\right)  \text{
such that}\\%
\begin{array}
[c]{rll}%
\left(  -\Delta-\kappa^{2}\right)  w & =0 & \text{in }\Omega^{+}\\
w & =\tilde{g} & \text{on }\partial\Omega\text{,}\\
\left\vert \left\langle \frac{\mathbf{x}}{\left\vert \mathbf{x}\right\vert
},\nabla w\left(  \mathbf{x}\right)  \right\rangle -\operatorname*{i}\kappa
w\left(  \mathbf{x}\right)  \right\vert  & =o\left(  \left\vert \mathbf{x}%
\right\vert ^{\frac{1-d}{2}}\right)  & \left\vert \mathbf{x}\right\vert
\rightarrow\infty
\end{array}
\end{array}
\]
has a unique weak solution. The mapping $\tilde{g}\mapsto w$ is called the
\textit{Steklov-Poincar\'{e} operator }and denoted by $S_{P}:H^{1/2}\left(
\partial\Omega\right)  \rightarrow H_{\operatorname*{loc}}^{1}\left(
\Omega^{+}\right)  $. The \textit{Dirichlet-to-Neumann map} is given by
$T_{\kappa}:=\gamma_{1}S_{P}:H^{1/2}\left(  \partial\Omega\right)  \rightarrow
H^{-1/2}\left(  \partial\Omega\right)  $, where $\gamma_{1}:=\partial
/\partial$\textbf{$n$} is the normal trace operator.

\begin{remark}
The heterogeneous Helmholtz problem \eqref{modelproblemDtN} is well-posed (cf.
\cite{Alessandrini2012,GrahamSauter2017,JerisonKenig85}).
\end{remark}

In this paper, we discuss the stability of problem \eqref{modelproblemDtN} for
$\Omega$ being the unit ball centred at the origin and for a wave speed $c$
that is piecewise constant on concentric layers. We will use an explicit
representation of the Green's operator to understand which parameter
configurations are well-behaved (i.e. where the stability constant is bounded
with respect to the frequency $\omega$) and which configuration lead to a
localisation effect (i.e. a stability constant that grows exponentially with
respect to $\omega$).

\subsection{Helmholtz Problem in Spherical Coordinates\label{subsec:spherical}%
}

For $d\geq2$, let $Y_{m,\mathbf{n}}$ denote the eigenfunctions (spherical
harmonics) of the (negative) Laplace-Beltrami operator $\Delta_{\Gamma}$ on
$\Gamma:= \partial\Omega$ (cf. \cite[\S 22]{Shubin2001})%
\begin{equation}
-\Delta_{\Gamma}Y_{m,\mathbf{n}}=\lambda_{m}Y_{m,\mathbf{n}},\quad
m\in\mathbb{N}_{0}\text{, }\mathbf{n}\in\iota_{m} \label{defiotam}%
\end{equation}
for some finite index set $\iota_{m}$ which corresponds to the multiplicity of
the eigenvalue $\lambda_{m}$. We assume that the eigenvalues $\lambda_{m}$ are
numbered such that%
\[
0=\lambda_{0}<\lambda_{1}<\ldots
\]
and the eigenfunctions form a orthonormal basis of $L^{2}\left(
\Gamma\right)  $. Explicitly it holds (cf. \cite[\S 22]{Shubin2001})
\begin{equation}
\lambda_{m}=m\left(  m+d-2\right)  \quad\text{with multiplicity }\sharp
\iota_{m}=\frac{2m+d-2}{m+d-2}\binom{m+d-2}{d-2}. \label{defiotamnumber}%
\end{equation}

We introduce spherical coordinates in $\mathbb{R}^{d}$ by $\mathbf{x}=r \boldsymbol{\xi}$,
for $r:=\left\vert \mathbf{x}\right\vert $ and $\boldsymbol{\xi}
:=\mathbf{x}/r\in\mathbb{S}_{d-1}$. This transformation is denoted by
$\mathbf{x}=\psi\left(  r,\boldsymbol{\xi}\right)  $ and for a 
function $w$ in Cartesian coordinates we write $\hat
{w}=w\circ\psi$. The Laplace operator in spherical coordinates is given by%
\[
\left(  \Delta v\right)  \circ\psi=\frac{1}{r^{d-1}}\partial_{r}\left(
r^{d-1}\partial_{r}\hat{v}\right)  +\frac{1}{r^{2}}\Delta_{\Gamma}\hat{v}.
\]

Next we transform the Helmholtz equation \eqref{modelproblemDtN} to spherical
coordinates and write $\hat{u}=u\circ\psi$ and $\hat{g}\left(  \boldsymbol{\xi
}\right)  =g\circ\psi\left(  1,\boldsymbol{\xi}\right)  $. We employ a Fourier
expansion to the boundary data%
\[
\hat{g}=\sum_{m\in\mathbb{N}_{0}}\sum_{\mathbf{n}\in\iota_{m}}\hat
{g}_{m,\mathbf{n}}Y_{m,\mathbf{n}}%
\]
and a similar expansion for the solution%
\begin{equation}
u\circ\psi\left(  r,\boldsymbol{\xi} \right)  =\hat{u}\left(  r,\boldsymbol{\xi}\right)  
=\sum_{m\in\mathbb{N}_{0}
}\sum_{\mathbf{n}\in\iota_{m}}\hat{u}_{m,\mathbf{n}}\left(  r\right)
Y_{m\mathbf{,n}}\left(  \boldsymbol{\xi}\right)  . \label{Ansatzu}%
\end{equation}
The ODE for the Fourier coefficients $\hat{u}_{m,\mathbf{n}}$ (more precisely
for the restrictions $\hat{u}_{m,\mathbf{n},j}:=\left.  \hat{u}_{m,\mathbf{n}%
}\right\vert _{\tau_{j}}$) reads for $r\in\tau_{j},$
\begin{equation}
-\frac{1}{r^{d-1}}\partial_{r}\left(  r^{d-1}\hat{u}_{m,\mathbf{n},j}^{\prime
}\left(  r\right)  \right)  +\left(  \frac{m\left(  m+d-2\right)  }{r^{2}%
}-\left(  \frac{\omega}{c_{j}}\right)  ^{2}\right)  \hat{u}_{m,\mathbf{n}%
,j}\left(  r\right)  =0. \label{locPDE}%
\end{equation}
Moreover we have the conditions at each interface
\begin{equation}
\left[  \hat{u}_{m,\mathbf{n}}\right]  _{x_{j}}=\left[  \hat{u}_{m,\mathbf{n}%
}^{\prime}\right]  _{x_{j}}=0,\qquad\qquad1\leq j\leq N-1, \label{locPDEjump}%
\end{equation}
where $\left[  f\right]  _{x}$ denotes the jump of $f$ at $x$.

\begin{remark}
[The case $d=1$]In one dimension, problem \eqref{modelproblemDtN} can be
written in a similar form. We set $\hat{u}\left(  r\right)  =u\left(
\left\vert x\right\vert \right)  $ and obtain for $\hat{u}_{j}:=
\hat{u}_{\mid_{\tau_{j}}}$ the ordinary differential equation%
\begin{equation}
-\hat{u}_{j}^{\prime\prime}-\left(  \frac{\omega}{c_{j}}\right)  ^{2}\hat
{u}_{j}=0\quad\text{in }\tau_{j} \label{ODE1D}%
\end{equation}
with jump conditions%
\[
\left[  \hat{u}\right]  _{x_{j}}=\left[  \hat{u}^{\prime}\right]  _{x_{j}%
}=0,\qquad\qquad1\leq j\leq N-1.
\]
This corresponds to \eqref{locPDE} for $d=1$ and $m=\mathbf{n}=0$ as well as
to \eqref{locPDEjump}. The function $u$ then is given by $u\left(  x\right)
=\hat{u}\left(  r\right)  Y_{0,0}\left(  \frac{x}{\left\vert x\right\vert
}\right)  $, where for $d=1$ we set $Y_{0,0}\left(  \pm1\right)  = 1$.
\end{remark}

We denote by

\begin{itemize}
\item $J_{\nu}$: the Bessel function of first kind and order $\nu$,

\item $j_{\nu}$: the spherical Bessel function of first kind and order $\nu$,

\item $H_{\nu}^{(1)}$: the Hankel function of first kind and order $\nu$,

\item $h_{\nu}^{(1)}$: the spherical Hankel function of first kind and order
$\nu$.
\end{itemize}

The fundamental system for the ODE \eqref{locPDE} are generated, for $d\geq2$,
by%
\[
f_{m,d,1}\left(  r\right)  :=\mathfrak{c}_{d}\frac{H_{_{m+\frac{d}{2}-1}%
}^{\left(  1\right)  }\left(  r\right)  }{r^{\frac{d}{2}-1}}\quad
\text{and\quad}f_{m,d,2}\left(  r\right)  :=\mathfrak{c}_{d}\frac
{J_{_{m+\frac{d}{2}-1}}\left(  r\right)  }{r^{\frac{d}{2}-1}}%
\]
and for the ODE \eqref{ODE1D} for $d=1$, by%
\[
f_{0,1,1}\left(  r\right)  :=\mathfrak{c}_{1}\operatorname*{e}%
\nolimits^{\operatorname*{i}r}\quad\text{and\quad}f_{0,d,2}\left(  r\right)
:=\mathfrak{c}_{1}\cos r,
\]
for some normalization constants $\mathfrak{c}_{d}>0$. For $d=1,2,3$, we set
\[
\mathfrak{c}_{1}:=1,\quad\mathfrak{c}_{2}:=1,\text{\quad}\mathfrak{c}%
_{3}:=\sqrt{\pi/2}%
\]
so that
\[%
\begin{array}
[c]{lll}%
f_{m,1,1}\left(  r\right)  :=\operatorname*{e}\nolimits^{\operatorname*{i}r} &
f_{m,d,2}\left(  r\right)  :=\cos r, & m = 0,\\
f_{m,2,1}\left(  r\right)  :=H_{m}^{\left(  1\right)  }\left(  r\right)  , &
f_{m,2,2}\left(  r\right)  :=J_{m}\left(  r\right)  , & m \in\mathbb{N}_{0}\\
f_{m,3,1}\left(  r\right)  :=h_{m}^{\left(  1\right)  }\left(  r\right)  , &
f_{m,3,2}\left(  r\right)  :=j_{m}\left(  r\right)  , & m \in\mathbb{N}_{0} .
\end{array}
\]

To include the case $d=1$ in the notation, we set $\mathcal{N}_{d}%
:=\mathbb{N}_{0}$ for $d\geq2$ and $\mathcal{N}_{1}:=\left\{  0\right\}  $.
Note that $\iota_{m}$ in \eqref{defiotam} depends on the dimension $d$ (cf.
\eqref{defiotamnumber}) and for $d=1$ we set $\iota_{0}:=\left\{  0\right\}
$. In this light, we write also short $\mathcal{N}$ for $\mathcal{N}_{d}$.
Also we set $\mathbb{S}_{0}=\{-1,1\}$ and for $d\geq2$ we denote by
$\mathbb{S}_{d-1}$ the unit sphere in $\mathbb{R}^{d}$.

The DtN boundary conditions in \eqref{modelproblemDtN} read (see, e.g.,
\cite[(3.7), (3.10), (3.25)]{MelenkSauterMathComp}):%
\[
\hat{u}_{m,\mathbf{n}}^{\prime}\left(  1\right)  -\frac{\omega}{c_{N}}%
\frac{f_{m,d,1}^{\prime}\left(  \frac{\omega}{c_{N}}\right)  }{f_{m,d,1}%
\left(  \frac{\omega}{c_{N}}\right)  }\hat{u}_{m,\mathbf{n}}\left(  1\right)
=\hat{g}_{m,\mathbf{n}}.
\]
Next, we impose an appropriate conditions at the origin which guarantees that
the solution $u$ in \eqref{Ansatzu} is smooth in at origin. We require for all
$\boldsymbol{\xi}\in\mathbb{S}_{d-1}$%
\begin{align*}
\lim_{r\searrow0}u\left(  r\boldsymbol{\xi}\right)   &  =\lim_{r\searrow0}u\left(  -r\boldsymbol{\xi}\right) \\
\lim_{r\searrow0}\frac{u\left(  r\boldsymbol{\xi}\right)  -u\left(  0\right)  }{r}  &  =\lim_{r\searrow0}\frac{u\left(
0\right)  -u\left(  -r\boldsymbol{\xi}\right)  }{r}%
\end{align*}
which is equivalent to%
\begin{align*}
\sum_{m\in\mathcal{N}}\sum_{\mathbf{n}\in\iota_{m}}\hat{u}_{m,\mathbf{n}%
}\left(  0\right)  Y_{m\mathbf{,n}}\left(  \boldsymbol{\xi}\right)   &
=\sum_{m\in\mathcal{N}}\sum_{\mathbf{n}\in\iota_{m}}\hat{u}_{m,\mathbf{n}%
}\left(  0\right)  Y_{m\mathbf{,n}}\left(  -\boldsymbol{\xi}\right) \\
\sum_{m\in\mathcal{N}}\sum_{\mathbf{n}\in\iota_{m}}\hat{u}_{m,\mathbf{n}%
}^{\prime}\left(  0\right)  Y_{m\mathbf{,n}}\left(  \boldsymbol{\xi}\right)
&  =-\sum_{m\in\mathcal{N}}\sum_{\mathbf{n}\in\iota_{m}}\hat{u}_{m,\mathbf{n}%
}^{\prime}\left(  0\right)  Y_{m\mathbf{,n}}\left(  -\boldsymbol{\xi}\right)
\end{align*}
We use that the spherical harmonics satisfy by definition (cf. e.g., \cite[p.
71]{Frye_spherical_harm})
\[
Y_{m,\mathbf{n}}\left(\boldsymbol{\xi}\right)  =\left(  -1\right)  ^{m}
Y_{m,\mathbf{n}}\left(  -\boldsymbol{\xi}\right)  .
\]
Hence,%
\begin{align*}
\hat{u}_{m,\mathbf{n}}\left(  0\right)   &  =\left(  -1\right)  ^{m}\hat
{u}_{m,\mathbf{n}}\left(  0\right) \\
\hat{u}_{m,\mathbf{n}}^{\prime}\left(  0\right)   &  =\left(  -1\right)
^{m+1}\hat{u}_{m,\mathbf{n}}^{\prime}\left(  0\right)  .
\end{align*}
This implies%
\begin{align*}
\hat{u}_{m,\mathbf{n}}\left(  0\right)   &  =0\quad\text{for odd }m\text{,}\\
\hat{u}_{m,\mathbf{n}}^{\prime}\left(  0\right)   &  =0\quad\text{for even }m
\end{align*}
and these are the boundary conditions at $r=0$.

In summary, for the considered radial symmetric case we study the following
system of ODEs:
\begin{equation}
-\frac{1}{r^{d-1}}\partial_{r}\left(  r^{d-1}\hat{u}_{m,\mathbf{n},j}^{\prime
}\left(  r\right)  \right)  +\left(  \frac{m\left(  m+d-2\right)  }{r^{2}%
}-\left(  \frac{\omega}{c_{j}}\right)  ^{2}\right)  \hat{u}_{m,\mathbf{n}%
,j}\left(  r\right)  =0, \label{SLV:homogeneoussystem}%
\end{equation}
for $r\in\tau_{j},1\leq j\leq N$, together with the interface conditions
\begin{equation}
\left[  \hat{u}_{m,\mathbf{n}}\right]  _{x_{j}}=\left[  \hat{u}_{m,\mathbf{n}%
}^{\prime}\right]  _{x_{j}}=0,\quad1\leq j\leq N-1, \label{eq:interfacecond}%
\end{equation}
and boundary conditions
\[%
\begin{split}
\hat{u}_{m,\mathbf{n}}^{\prime}\left(  1\right)  -\frac{\omega}{c_{N}}%
\frac{f_{m,d,1}^{\prime}\left(  \frac{\omega}{c_{N}}\right)  }{f_{m,d,1}%
\left(  \frac{\omega}{c_{N}}\right)  }\hat{u}_{m,\mathbf{n}}\left(  1\right)
=  &  \hat{g}_{m,\mathbf{n}},\\
\hat{u}_{m,\mathbf{n}}\left(  0\right)  =  &  0\quad\text{for odd }m\text{,}\\
\hat{u}_{m,\mathbf{n}}^{\prime}\left(  0\right)  =  &  0\quad\text{for even
}m.
\end{split}
\]
For the solution of the homogenous equation \eqref{SLV:homogeneoussystem} we
employ the ansatz for $m\in\mathcal{N}$ and $\mathbf{n}\in\iota_{m}$
\begin{equation}
\hat{u}_{m,\mathbf{n}}|_{\tau_{j}}=A_{\left(  m,\mathbf{n}\right)
,j}f_{m,d,1}\left(  \frac{\omega}{c_{j}}\cdot\right)  +B_{\left(
m,\mathbf{n}\right)  ,j}f_{m,d,2}\left(  \frac{\omega}{c_{j}}\cdot\right)  .
\label{ansatzontauj}%
\end{equation}
The boundary conditions at the origin and the boundary condition at $1$
imply\footnote{The Wronskian $\mathcal{W}\left(  \varphi_{1},\varphi
_{2}\right)  $ of two functions $\varphi_{1}$ and $\varphi_{2}$ is given by%
\[
\mathcal{W}\left(  \varphi_{1},\varphi_{2}\right)  \left(  z\right)
:=\varphi_{1}\left(  z\right)  \varphi_{2}^{\prime}\left(  z\right)
-\varphi_{1}^{\prime}\left(  z\right)  \varphi_{2}\left(  z\right)  .
\]
}
\begin{subequations}
\label{eq:globboundarycond}%
\begin{align}
A_{\left(  m,\mathbf{n}\right)  ,1}  &  =0\label{Amn1}\\
B_{\left(  m,\mathbf{n}\right)  ,N}  &  =\frac{f_{m,d,1}\left(  \kappa
_{N,N}\right)  }{\kappa_{N,N}\mathcal{W}\left(  f_{m,d,1},f_{m,2}\right)
\left(  \kappa_{N,N}\right)  }\hat{g}_{m,\mathbf{n}} \label{BmnN}%
\end{align}
with
\end{subequations}
\[
z_{\ell}:=\omega x_{\ell}\quad\text{and\quad}\kappa_{j,\ell}:=z_{\ell}/c_{j}.
\]

The interface conditions \eqref{eq:interfacecond} can be rewritten as a system
of linear equation given by
\[%
\begin{split}
&  A_{\left(  m,\mathbf{n}\right)  ,\ell}f_{m,d,1}\left(  \kappa_{\ell,\ell
}\right)  +B_{\left(  m,\mathbf{n}\right)  ,\ell}f_{m,d,2}\left(  \kappa
_{\ell,\ell}\right)  =\\
&  A_{\left(  m,\mathbf{n}\right)  ,\ell+1}f_{m,d,1}\left(  \kappa
_{\ell+1,\ell}\right)  +B_{\left(  m,\mathbf{n}\right)  ,\ell+1}%
f_{m,d,2}\left(  \kappa_{\ell+1,\ell}\right)  ,\\
&  \frac{A_{\left(  m,\mathbf{n}\right)  ,\ell}}{c_{\ell}}f_{m,d,1}^{\prime
}\left(  \kappa_{\ell,\ell}\right)  +\frac{B_{\left(  m,\mathbf{n}\right)
,\ell}}{c_{\ell}}f_{m,d,2}^{\prime}\left(  \kappa_{\ell,\ell}\right)  =\\
&  \frac{A_{\left(  m,\mathbf{n}\right)  ,\ell+1}}{c_{\ell+1}}f_{m,d,1}%
^{\prime}\left(  \kappa_{\ell+1,\ell}\right)  +\frac{B_{\left(  m,\mathbf{n}%
\right)  ,\ell+1}}{c_{\ell+1}}f_{m,d,2}^{\prime}\left(  \kappa_{\ell+1,\ell
}\right)  ,
\end{split}
\]
for $1\leq\ell\leq N-1$, whose solution is given by
\begin{equation}
\mathbf{x}_{m,\mathbf{n}}=\mathbf{M}_{m,\mathbf{n}}^{\text{Green}}%
\mathbf{r}_{m,\mathbf{n}}, \label{Greenrep}%
\end{equation}
where
\[
\mathbf{x}_{m,\mathbf{n}}=\left(  B_{\left(  m,\mathbf{n}\right)
,1},A_{\left(  m,\mathbf{n}\right)  ,1},B_{\left(  m,\mathbf{n}\right)
,2},\cdots,A_{\left(  m,\mathbf{n}\right)  ,n-1},B_{\left(  m,\mathbf{n}%
\right)  ,N-1},A_{\left(  m,\mathbf{n}\right)  ,N}\right)  ^{\intercal},
\]
supplemented by $A_{\left(  m,\mathbf{n}\right)  ,1},~B_{\left(
m,\mathbf{n}\right)  ,N}$ as in \eqref{eq:globboundarycond} contains the
coefficients in \eqref{ansatzontauj} and $\mathbf{M}_{m,\mathbf{n}%
}^{\text{Green}}\in\mathbb{C}^{2n\times2n}$, $\mathbf{r}_{m,\mathbf{n}}%
\in\mathbb{C}^{2n}$. The precise definition of the linear system corresponding
to \eqref{Greenrep} is given in \eqref{fullinteriorsysfinal}.

\section{Representation of the Green's Operator and Stability Estimate\label{sec:RepGreen}}

In this section, we introduce a representation of the radial Green's operator
and formulate the main stability estimate.

\subsection{The Radial Green's Operator}

In this section, we rewrite the linear system \eqref{eq:interfacecond} in a
suitable way and find a representation of the entries of the Green's operator
$\mathbf{M}_{m,\mathbf{n}}^{\text{Green}}$.

\begin{notation}
In the following, we skip the indices $\mathbf{n}$ and $d$ and write short
$f_{m,\ell}$ for $f_{m,d,\ell}$ and $A_{m,j}$ for $A_{\left(  m,\mathbf{n}%
\right)  ,j}$ and similar for other quantities. Using the definition of the
Wronskian $\mathcal{W}(\cdot,\cdot)$, we define
\begin{align*}
w_{m,j,k,\ell}^{p,q}:=  &  \mathcal{W}\left(  f_{m,p}\left(  \frac{\cdot
}{c_{j}}\right)  ,f_{m,q}\left(  \frac{\cdot}{c_{k}}\right)  \right)  \left(
z_{\ell}\right) \\
=  &  \frac{f_{m,p}\left(  \kappa_{j,\ell}\right)  f_{m,q}^{\prime}\left(
\kappa_{k,\ell}\right)  }{c_{k}}-\frac{f_{m,p}^{\prime}\left(  \kappa_{j,\ell
}\right)  f_{m,q}\left(  \kappa_{k,\ell}\right)  }{c_{j}}.
\end{align*}

\end{notation}

The interface conditions \eqref{eq:interfacecond} together with
\eqref{eq:globboundarycond} give%
\begin{equation}
\left[
\begin{array}
[c]{ccccc}%
\mathbf{S}_{m}^{\left(  1\right)  } & \mathbf{T}_{m}^{\left(  1\right)  } &
\mathbf{0} & \dots & \mathbf{0}\\
\mathbf{R}_{m}^{\left(  1\right)  } & \mathbf{S}_{m}^{\left(  2\right)  } &
\ddots & \ddots & \vdots\\
\mathbf{0} & \ddots & \ddots &  & \mathbf{0}\\
\vdots & \ddots &  &  & \mathbf{T}_{m}^{\left(  n-1\right)  }\\
\mathbf{0} & \dots & \mathbf{0} & \mathbf{R}_{m}^{\left(  n-1\right)  } &
\mathbf{S}_{m}^{\left(  n\right)  }%
\end{array}
\right]  \left(
\begin{array}
[c]{c}%
B_{1}\\
\vdots\\
A_{n}\\
B_{n}\\
A_{n+1}%
\end{array}
\right)  =C\left(
\begin{array}
[c]{c}%
0\\
\vdots\\
0\\
f_{m,2}\left(  \kappa_{n+1,n}\right) \\
\frac{f_{m,2}^{\prime}\left(  \kappa_{n+1,n}\right)  }{c_{n+1}}%
\end{array}
\right)  \label{matSRT}%
\end{equation}
with%
\[
n:=N-1,\qquad C:=\frac{f_{m,1}\left(  \kappa_{n+1,n+1}\right)  g_{m}}%
{\kappa_{n+1,n+1}\mathcal{W}\left(  f_{m,1},f_{m,2}\right)  \left(
\kappa_{n+1,n+1}\right)  }%
\]
and
\begin{align*}
\mathbf{R}_{m}^{\left(  \ell\right)  }  &  =\left[  \mathbf{0}\mid
\mathbf{r}_{m}^{\left(  \ell\right)  }\right]  =\left[
\begin{array}
[c]{ll}%
0 & f_{m,1}\left(  \kappa_{\ell+1,\ell+1}\right) \\
0 & \frac{f_{m,1}^{\prime}\left(  \kappa_{\ell+1,\ell+1}\right)  }{c_{\ell+1}}%
\end{array}
\right]  \quad1\leq\ell\leq n-1,\\
\mathbf{S}_{m}^{\left(  \ell\right)  }  &  =\left[  \mathbf{s}_{m,1}^{\left(
\ell\right)  }\mid\mathbf{s}_{m,2}^{\left(  \ell\right)  }\right]  =\left[
\begin{array}
[c]{ll}%
f_{m,2}\left(  \kappa_{\ell,\ell}\right)  & -f_{m,1}\left(  \kappa
_{\ell+1,\ell}\right) \\
\frac{f_{m,2}^{\prime}\left(  \kappa_{\ell,\ell}\right)  }{c_{\ell}} &
-\frac{f_{m,1}^{\prime}\left(  \kappa_{\ell+1,\ell}\right)  }{c_{\ell+1}}%
\end{array}
\right]  \quad1\leq\ell\leq n,\\
\mathbf{T}_{m}^{\left(  \ell\right)  }  &  =\left[  \mathbf{t}_{m}^{\left(
\ell\right)  }\mid\mathbf{0}\right]  =\left[
\begin{array}
[c]{ll}%
-f_{m,2}\left(  \kappa_{\ell+1,\ell}\right)  & 0\\
-\frac{f_{m,2}^{\prime}\left(  \kappa_{\ell+1,\ell}\right)  }{c_{\ell+1}} & 0
\end{array}
\right]  \quad1\leq\ell\leq n-1.
\end{align*}

Next we transform this system to a tri-diagonal system and first introduce
some notation. We define the block-diagonal matrix%
\[
\mathbf{D}_{m}:=\operatorname{diag}\left[  \mathbf{D}_{m}^{\left(
\ell\right)  }:1\leq\ell\leq n\right]  ,
\]
for\footnote{We set $%
\begin{pmatrix}
a\\
b
\end{pmatrix}
^{\perp}:=(b,-a)$.}%
\[
\mathbf{D}_{m}^{\left(  \ell\right)  }=\frac{1}{w_{m,\ell+1,\ell,\ell}^{2,1}%
}\left[
\begin{array}
[c]{r}%
\left(  \mathbf{t}_{m}^{\ell}\right)  ^{\perp}\\
\left(  \mathbf{r}_{m}^{\ell-1}\right)  ^{\perp}%
\end{array}
\right]  ,\qquad1\leq\ell\leq n,
\]
and multiply \eqref{matSRT} from left with $\mathbf{D}_{m}$ to obtain%
\begin{equation}
\mathbf{\hat{M}}_{m}^{\left(  2n\right)  }\left(
\begin{array}
[c]{c}%
B_{1}\\
\vdots\\
A_{N-1}\\
B_{N-1}\\
A_{N}%
\end{array}
\right)  =\frac{f_{m,1}\left(  \kappa_{N,N}\right)  }{\omega w_{m,N,N,N}%
^{1,2}}\left(
\begin{array}
[c]{c}%
0\\
\vdots\\
0\\
g_{m}%
\end{array}
\right)  \label{fullinteriorsysfinal}%
\end{equation}
with%
\begin{equation}
\mathbf{\hat{M}}_{m}^{\left(  2n\right)  }:=\left[
\begin{array}
[c]{ccccc}%
\mathbf{\hat{S}}_{m}^{\left(  1\right)  } & \mathbf{\hat{T}}_{m}^{\left(
1\right)  } & \mathbf{0} & \dots & \mathbf{0}\\
\mathbf{\hat{R}}_{m}^{\left(  1\right)  } & \mathbf{\hat{S}}_{m}^{\left(
2\right)  } & \ddots & \ddots & \vdots\\
\mathbf{0} & \ddots & \ddots &  & \mathbf{0}\\
\vdots & \ddots &  &  & \mathbf{\hat{T}}_{m}^{\left(  n-1\right)  }\\
\mathbf{0} & \dots & \mathbf{0} & \mathbf{\hat{R}}_{m}^{\left(  n-1\right)  }
& \mathbf{\hat{S}}_{m}^{\left(  n\right)  }%
\end{array}
\right]  \label{defMhat2n}%
\end{equation}
and%
\begin{align*}
\mathbf{\hat{R}}_{m}^{\left(  \ell\right)  }  &  :=\mathbf{D}_{m}^{\left(
\ell+1\right)  }\mathbf{R}_{m}^{\left(  \ell\right)  }=\left[
\begin{array}
[c]{rr}%
0 & 1\\
0 & 0
\end{array}
\right]  ,\\
\mathbf{\hat{S}}_{m}^{\left(  \ell\right)  }  &  :=\mathbf{D}_{m}^{\left(
\ell\right)  }\mathbf{S}_{m}^{\left(  \ell\right)  }=\frac{1}{w_{m,\ell
+1,\ell,\ell}^{2,1}}\left[
\begin{array}
[c]{rr}%
w_{m,\ell+1,\ell,\ell}^{2,2} & w_{m,\ell+1,\ell+1,\ell}^{1,2}\\
-w_{m,\ell,\ell,\ell}^{1,2} & w_{m,\ell,\ell+1,\ell}^{1,1}%
\end{array}
\right]  ,\\
\mathbf{\hat{T}}_{m}^{\left(  \ell\right)  }  &  :=\mathbf{D}_{m}^{\left(
\ell\right)  }\mathbf{T}_{m}^{\left(  \ell\right)  }=\left[
\begin{array}
[c]{rr}%
0 & 0\\
-1 & 0
\end{array}
\right]  .
\end{align*}

For later use, we define the matrix $\mathbf{M}^{\left(  \ell\right)  }$ for
odd $\ell\ $by removing the last row and column in $\mathbf{M}^{\left(
\ell+1\right)  }$.

The Green's operator is given by $\mathbf{M}_{m}^{\text{Green}}:=\left(
\mathbf{\hat{M}}_{m}^{\left(  2n\right)  }\right)  ^{-1}$ and the right-hand
side is defined by $\mathbf{r}=\left(  0,\dots,0,\frac{f_{m,1}\left(
\kappa_{N,N}\right)  }{\omega w_{m,N,N,N}^{1,2}}g_{m}\right)  $. Note that
\[
\mathcal{W}\left(  f_{m,d,1},f_{m,d,2}\right)  \left(  r\right)
=\left\{
\begin{array}
[c]{cc}%
-\operatorname*{i} & d=1,\\
-\operatorname*{i}\frac{2}{\pi}\mathfrak{c}_{d}^{2}\frac{1}{r^{d-1}} & d\geq2.
\end{array}
\right.
\]
so that%
\begin{equation}
w_{m,N,N,N}^{1,2}=-\frac{\operatorname*{i}}{c_{N}}\left(  \frac{c_{N}}{\omega
}\right)  ^{d-1}\times\left\{
\begin{array}
[c]{cc}%
1 & d=1,\\
\frac{2}{\pi}\mathfrak{c}_{d}^{2} & d\geq2.
\end{array}
\right.  \label{wmNNNcomput}%
\end{equation}
We use Cramer's rule to represent the entries of the Green's operator and
derive a recursive expression for the arising quotient of determinants (see
Theorem \ref{th:final_rep}). For $1\leq\ell\leq n=N-1$, we define the
quantities%

\begin{subequations}
\label{def:gammaqnew}
\begin{align}
\tilde{\gamma}_{m,\ell}^{+}  &  :=\frac{f_{m,1}\left(  \kappa_{\ell+1,\ell
}\right)  \overline{f_{m,1}^{\prime}\left(  \kappa_{\ell,\ell}\right)  }%
}{c_{\ell}}-\frac{f_{m,1}^{\prime}\left(  \kappa_{\ell+1,\ell}\right)
\overline{f_{m,1}\left(  \kappa_{\ell,\ell}\right)  }}{c_{\ell+1}}, \label{def:gammaqnewa}\\
\tilde{\gamma}_{m,\ell}^{-}  &  :=\frac{f_{m,1}^{\prime}\left(  \kappa
_{\ell,\ell}\right)  f_{m,1}\left(  \kappa_{\ell+1,\ell}\right)  }{c_{\ell}%
}-\frac{f_{m,1}^{\prime}\left(  \kappa_{\ell+1,\ell}\right)  f_{m,1}\left(
\kappa_{\ell,\ell}\right)  }{c_{\ell+1}}, \label{def:gammaqnewb}\\
\tilde{q}_{m,\ell}  &  :=\frac{\tilde{\gamma}_{m,\ell}^{-}}{\tilde{\gamma
}_{m,\ell}^{+}}. \label{def:gammaqq}%
\end{align}
\end{subequations}

We note that \eqref{def:gammaqq} is well defined since $\tilde{\gamma}%
_{m,\ell}^{+}\neq0$ (cf. proof of \cite[Lemma 1]{HansenPoignardVogelius2007}).
We also remark that the fundamental solution $f_{m,1}(x)$ scales with the
prefactor $e^{\operatorname*{i}x}$ as $x\rightarrow\infty$ (cf.
\cite{NIST:DLMF}). Therefore, it is natural to introduce
\[
\gamma_{m,\ell}^{\pm}:=\operatorname*{i}\operatorname{e}^{\operatorname*{i}%
\left(  \pm\frac{z_{\ell}}{c_{\ell}}-\frac{z_{\ell}}{c_{\ell+1}}\right)
}\tilde{\gamma}_{m,\ell}^{\pm},\qquad q_{m,\ell}:=\frac{\gamma_{m,\ell}^{-}%
}{\gamma_{m,\ell}^{+}}%
\]
and the sequence $\beta_{m,\ell}$ for $0\leq\ell\leq n$ recursively by%
\begin{equation}%
\begin{split}
\beta_{m,0}  &  :=1,\\
\beta_{m,\ell}  &  :=\frac{\gamma_{m,\ell}^{+}}{2\operatorname*{i}%
w_{m,\ell+1,\ell+1,\ell}^{1,2}}\left(  \operatorname{e}^{-\operatorname*{i}%
\frac{\omega h_{\ell}}{c_{\ell}}}\beta_{m,\ell-1}+q_{m,\ell}\overline
{\operatorname{e}^{-\operatorname*{i}\omega\frac{h_{\ell}}{c_{\ell}}}%
\beta_{m,\ell-1}}\right)  .
\end{split}
\label{recursioncheck}%
\end{equation}

\begin{remark}
\label{RemLinRec}Recursion \eqref{recursioncheck} can be written as a linear
recursion for the real and imaginary part of $\beta_{m,\ell}=\beta_{m,\ell
}^{R}+\operatorname*{i}\beta_{m,\ell}^{\operatorname*{I}}$. We set%
\begin{align*}
\theta_{m,\ell}  &  :=\frac{\gamma_{m,\ell}^{+}}{2\operatorname*{i}%
w_{m,\ell+1,\ell+1,\ell}^{1,2}}\operatorname{e}^{-\operatorname*{i}%
\frac{\omega h_{\ell}}{c_{\ell}}}=:\theta_{m,\ell}^{R}+\operatorname*{i}%
\theta_{m,\ell}^{I}\\
\phi_{m,\ell}  &  :=\frac{\gamma_{m,\ell}^{+}}{2\operatorname*{i}%
w_{m,\ell+1,\ell+1,\ell}^{1,2}}q_{m,\ell}\operatorname{e}^{\operatorname*{i}%
\omega\frac{h_{\ell}}{c_{\ell}}}=:\phi_{m,\ell}^{R}+\operatorname*{i}%
\phi_{m,\ell}^{I}%
\end{align*}
and obtain $\left(  \beta_{m,0}^{R},\beta_{m,0}^{I}\right)  =\left(
1,0\right)  $ and
\[
\left(
\begin{array}
[c]{c}%
\beta_{m,\ell}^{R}\\
\beta_{m,\ell}^{I}%
\end{array}
\right)  =\left[
\begin{array}
[c]{rr}%
\theta_{m,\ell}^{R}+\phi_{m,\ell}^{R} & \phi_{m,\ell}^{I}-\theta_{m,\ell}%
^{I}\\
\theta_{m,\ell}^{I}+\phi_{m,\ell}^{I} & \theta_{m,\ell}^{R}-\phi_{m,\ell}^{R}%
\end{array}
\right]  \left(
\begin{array}
[c]{c}%
\beta_{m,\ell-1}^{R}\\
\beta_{m,\ell-1}^{I}%
\end{array}
\right)
\]

\end{remark}

\begin{remark}
\label{rmk:wronskian} The Wronskian $w_{m,\ell+1,\ell+1,\ell}^{1,2}$ is
independent of $m$, but depends on the dimension $d$. Moreover it holds for
$1\leq\ell\leq n$%
\[
\operatorname*{i}w_{m,\ell+1,\ell+1,\ell}^{1,2}=\frac{c_{\ell+1}^{d-2}%
}{z_{\ell}^{d-1}}\times\left\{
\begin{array}
[c]{ll}%
1 & d=1\\
\frac{2}{\pi}\mathfrak{c}_{d}^{2} & d\geq2
\end{array}
\right.
\]
for all dimensions $d$. In particular, $\operatorname*{i}w_{m,\ell
+1,\ell+1,\ell}^{1,2}\in\mathbb{R}_{>0}$.
\end{remark}

\begin{theorem}
\label{th:final_rep} Let $\beta_{m,n,\ell}$ be recursively defined as in
\eqref{recursioncheck}. Then the entries of the Green's operator are given by%
\begin{equation}%
\begin{split}
\left(  \mathbf{M}_{m}^{\operatorname{Green}}\right)  _{2\ell,2n}  &
=-\left(  \frac{\operatorname{Im}\left(  \operatorname{e}^{\operatorname*{i}%
\frac{z_{\ell}}{c_{\ell+1}}}\beta_{m,\ell}\right)  }{\operatorname{e}%
^{\operatorname*{i}\frac{z_{n}}{c_{n+1}}}\beta_{m,n}}\right)  ,\\
\left(  \mathbf{M}_{m}^{\operatorname{Green}}\right)  _{2\ell-1,2n}  &
=-\left(  \frac{\operatorname{e}^{\operatorname*{i}\frac{z_{\ell-1}}{c_{\ell}%
}}\beta_{m,\ell-1}}{\operatorname{e}^{\operatorname*{i}\frac{z_{n}}{c_{n+1}}%
}\beta_{m,n}}\right)  ,
\end{split}
\label{eq:repMGreen}%
\end{equation}
for $1\leq\ell\leq n$.
\end{theorem}

The proof of Theorem \ref{th:final_rep} is postponed to Section
\ref{sec:proofRep}.

\subsection{Representation for the Fourier Mode $m=0$}

\label{Ch:Repm0} In this section, we restrict to $d=3$ and $m=0$. Then the
variables in the definition of the sequence $\beta_{0,\ell}$ simplify in the
following way%
\begin{align*}
\gamma_{0,\ell}^{+}  &  =\frac{c_{\ell+1}+c_{\ell}}{z_{\ell}^{2}},\quad
\gamma_{0,\ell}^{-}=\frac{c_{\ell+1}-c_{\ell}}{z_{\ell}^{2}},\\
q_{\ell}:=q_{0,\ell}  &  =\frac{c_{\ell+1}-c_{\ell}}{c_{\ell+1}+c_{\ell}%
},\qquad2\operatorname*{i}w_{m,k+1,k+1,k}^{1,2}=\frac{2c_{k+1}}{z_{k}^{2}},\\
\frac{2c_{k+1}}{z_{k}^{2}\gamma_{0,k}^{+}}  &  =\frac{2c_{k+1}}{c_{k}+c_{k+1}%
}=1+q_{k}.
\end{align*}
The recursion for $\beta_{0,\ell}$ is then given by%
\begin{equation}%
\begin{array}
[c]{l}%
\beta_{0,0}:=1,\\
\beta_{0,\ell}:=\dfrac{1}{1+q_{\ell}}\left(  \operatorname{e}%
^{-\operatorname*{i}\frac{\omega h_{\ell}}{c_{\ell}}}\beta_{0,\ell-1}+q_{\ell
}\overline{\operatorname{e}^{-\operatorname*{i}\omega\frac{h_{\ell}}{c_{\ell}%
}}\beta_{0,\ell-1}}\right)  .
\end{array}
\label{eq:recursion2m0}%
\end{equation}

\subsection{Main Stability Theorem\label{subsec:mainstability}}

First, we express the $\left\Vert \cdot\right\Vert _{\mathcal{H}}$-norm of the
function $u$ in \eqref{Ansatzu} in spherical coordinates and recall well-known
facts from spherical harmonics: The gradient in spherical coordinates is given
by%
\[
\left(  \nabla v\right)  \circ\psi=\partial_{r}\hat{v}\overrightarrow{e_{r}%
}+\frac{1}{r}\nabla_{\Gamma}\hat{v},
\]
where $\nabla_{\Gamma}$ denotes the surface gradient and $\overrightarrow
{e_{r}}$ the basis vector in $r$-direction. For the spherical harmonics
$Y_{m,\mathbf{n}}$ it holds%
\begin{align*}
\left(  Y_{m,\mathbf{n}},Y_{m^{\prime},\mathbf{n}^{\prime}}\right)  _{\Gamma
}&=\delta_{m,m^{\prime}}\delta_{\mathbf{n},\mathbf{n}^{\prime}}, 
\\
\left( \nabla_{\Gamma}Y_{m,\mathbf{n}},\nabla_{\Gamma}Y_{m,\mathbf{n}}\right)
_{\Gamma}&=\lambda_{m}\delta_{m,m^{\prime}}\delta_{\mathbf{n},\mathbf{n} 
^{\prime}},\\
\left(  \nabla_{\Gamma}Y_{m,\mathbf{n}},\overrightarrow{e_{r}}\right)  _{\Gamma}&=0. 
\end{align*}
Hence, the transformation rule of variables yield (with $\mathfrak{g}_{\nu
}\left(  r\right)  :=r^{\left(  d-1\right)  /2-\nu}$ and $I=\left(
0,1\right)  $)%
\begin{align*}
&\left\Vert u\right\Vert _{\mathcal{H}}^{2}=    \left\Vert \nabla u\right\Vert
^{2}+\left\Vert \frac{\omega}{c}u\right\Vert ^{2}\\
&=    \sum_{m\in\mathcal{N}}\sum_{\mathbf{n}\in\iota_{m}}\left(  \left\Vert
\hat{u}_{m,\mathbf{n}}^{\prime}\mathfrak{g}_{0}\right\Vert _{L^{2}\left(
I\right)  }^{2}\left\Vert Y_{m,\mathbf{n}}\right\Vert _{L^{2}\left(
\mathbb{S}_{d-1}\right)  }^{2}+\left\Vert \hat{u}_{m,\mathbf{n}} 
\mathfrak{g}_{1}\right\Vert _{L^{2}\left(  I\right)  }^{2}\left\Vert
\nabla_{\Gamma}Y_{m,\mathbf{n}}\right\Vert _{L^{2}\left(  \mathbb{S}%
_{d-1}\right)  }^{2}\right. \\
& \phantom{=\sum_{m\in\mathcal{N}}\sum_{\mathbf{n}\in\iota_{m}}+} \left.  +\left\Vert \frac{\omega}{c}\hat{u}_{m,\mathbf{n}}\mathfrak{g}%
_{0}\right\Vert _{L^{2}\left(  I\right)  }^{2}\left\Vert Y_{m,\mathbf{n}%
}\right\Vert _{L^{2}\left(  \mathbb{S}_{d-1}\right)  }^{2}\right) \\
&=    \sum_{m\in\mathcal{N}}\sum_{\mathbf{n}\in\iota_{m}}\left(  \left\Vert
\hat{u}_{m,\mathbf{n}}^{\prime}\mathfrak{g}_{0}\right\Vert _{L^{2}\left(
I\right)  }^{2}+\left\Vert \frac{\omega}{c}\hat{u}_{m,\mathbf{n}}%
\mathfrak{g}_{0}\right\Vert _{L^{2}\left(  I\right)  }^{2}+\lambda
_{m}\left\Vert \hat{u}_{m,\mathbf{n}}\mathfrak{g}_{1}\right\Vert
_{L^{2}\left(  I\right)  }^{2}\right)  .
\end{align*}
We split the integral over $I$ into a sum of integrals over $\tau_{j}$ and
employ ansatz \eqref{ansatzontauj} to get the estimate%
\begin{align*}
&\left\Vert u\right\Vert _{\mathcal{H}}^{2}    
\leq2\sum_{m\in\mathcal{N}} \sum_{\mathbf{n}\in\iota_{m}}\sum_{j=1}^{N}
\left(  \left(  \frac{\omega}{c_{j}}A_{\left(  m,\mathbf{n}\right)  ,j}\right)  ^{2}\left\Vert
f_{m,1}^{\prime}\left(  \frac{\omega}{c_{j}}\cdot\right)  \mathfrak{g}_{0}\right\Vert _{L^{2}\left(  \tau_{j}\right)  }^{2}\right.
\\
& \left.+\left(  \frac{\omega
}{c_{j}}B_{\left(  m,\mathbf{n}\right)  ,j}\right)  ^{2}\left\Vert
f_{m,2}^{\prime}\left(  \frac{\omega}{c_{j}}\cdot\right)  \mathfrak{g}%
_{0}\right\Vert _{L^{2}\left(  \tau_{j}\right)  }^{2}\right.  \\
& +\left.  \left(  \frac{\omega}{c_{j}}A_{\left(  m,\mathbf{n}\right)
,j}\right)  ^{2}\left\Vert f_{m,1}\left(  \frac{\omega}{c_{j}}\cdot\right)
\mathfrak{g}_{0}\right\Vert _{L^{2}\left(  \tau_{j}\right)  }^{2}+\left(
\frac{\omega}{c_{j}}B_{\left(  m,\mathbf{n}\right)  ,j}\right)  ^{2}\left\Vert
f_{m,2}\left(  \frac{\omega}{c_{j}}\cdot\right)  \mathfrak{g}_{0}\right\Vert
_{L^{2}\left(  \tau_{j}\right)  }^{2}\right. \\
&  +\left.  \lambda_{m}\left(  A_{\left(  m,\mathbf{n}\right)  ,j}%
^{2}\left\Vert f_{m,1}\left(  \frac{\omega}{c_{j}}\cdot\right)  \mathfrak{g}%
_{1}\right\Vert _{L^{2}\left(  \tau_{j}\right)  }^{2}+B_{\left(
m,\mathbf{n}\right)  ,j}^{2}\left\Vert f_{m,2}\left(  \frac{\omega}{c_{j}%
}\cdot\right)  \mathfrak{g}_{1}\right\Vert _{L^{2}\left(  \tau_{j}\right)
}^{2}\right)  \right)  .
\end{align*}
If $m=0$, we have $\lambda_{m}=0$ and the term $\lambda_{m}\left(
\ldots\right)  $ vanishes.

We set $g_{m,\mathbf{n}}=0$ for all $m\in\mathcal{N}$ and $\mathbf{n}\in
\iota_{m}$ with $\left(  m,\mathbf{n}\right)  \neq\left(  0,\mathbf{0}\right)
$. Then, the solution $u$ in the form \eqref{Ansatzu} satisfies
$u_{m,\mathbf{n}}=0$ for all $m\in\mathcal{N}$ and $\mathbf{n}\in\iota_{m}$
with $\left(  m,\mathbf{n}\right)  \neq\left(  0,\mathbf{0}\right)  $. We set
$A_{j}:=A_{\left(  0,\mathbf{0}\right)  ,j}$ and $B_{j}:=B_{\left(
0,\mathbf{0}\right)  ,j}$ and get%
\begin{align}
\left\Vert u\right\Vert _{\mathcal{H}}^{2}  &  \leq2\sum_{j=1}^{N}\left(
\left(  \frac{\omega}{c_{j}}A_{j}\right)  ^{2}\left\Vert f_{0,1}^{\prime
}\left(  \frac{\omega}{c_{j}}\cdot\right)  \mathfrak{g}_{0}\right\Vert
_{L^{2}\left(  \tau_{j}\right)  }^{2}+\left(  \frac{\omega}{c_{j}}%
B_{j}\right)  ^{2}\left\Vert f_{0,2}^{\prime}\left(  \frac{\omega}{c_{j}}%
\cdot\right)  \mathfrak{g}_{0}\right\Vert _{L^{2}\left(  \tau_{j}\right)
}^{2}\right. \label{eq:gradl2estsol}\\
&  +\left.  \left(  \frac{\omega}{c_{j}}A_{j}\right)  ^{2}\left\Vert
f_{0,1}\left(  \frac{\omega}{c_{j}}\cdot\right)  \mathfrak{g}_{0}\right\Vert
_{L^{2}\left(  \tau_{j}\right)  }^{2}+\left(  \frac{\omega}{c_{j}}%
B_{j}\right)  ^{2}\left\Vert f_{0,2}\left(  \frac{\omega}{c_{j}}\cdot\right)
\mathfrak{g}_{0}\right\Vert _{L^{2}\left(  \tau_{j}\right)  }^{2}\right)
\label{eq:l2estsol}%
\end{align}

We are also interested in pointwise estimates. For $r\in\tau_{j}$ we have in
the considered case (note that $\nabla_{\Gamma}Y_{0\mathbf{,0}}\left(
\boldsymbol{\xi}\right)  =0$ and $\left\vert Y_{0\mathbf{,0}}\left(
\boldsymbol{\xi}\right)  \right\vert =\left\vert \Gamma\right\vert ^{-1/2}$
with the surface measure $\left\vert \Gamma\right\vert $ of $\mathbb{S}_{d-1}%
$)%
\begin{align*}
\left\vert \hat{u}\left(  r,\boldsymbol{\xi}\right)  \right\vert  &
\leq\left\vert \Gamma\right\vert ^{-1/2}\left(  \left\vert A_{j}\right\vert
\left\vert f_{0,1}\left(  \frac{\omega}{c_{j}}r\right)  \right\vert
+\left\vert B_{j}\right\vert \left\Vert f_{0,2}\left(  \frac{\omega}{c_{j}%
}\cdot\right)  \right\Vert _{L^{\infty}\left(  \tau_{j}\right)  }\right)  ,\\
\left\vert \left(  \nabla u\right)  \circ\psi\left(  r,\boldsymbol{\xi
}\right)  \right\vert  &  \leq\left\vert \Gamma\right\vert ^{-1/2}\left(
\left\vert A_{j}\right\vert \left\vert f_{0,1}^{\prime}\left(  \frac{\omega
}{c_{j}}r\right)  \right\vert +\left\vert B_{j}\right\vert \left\Vert
f_{0,2}^{\prime}\left(  \frac{\omega}{c_{j}}\cdot\right)  \right\Vert
_{L^{\infty}\left(  \tau_{j}\right)  }\right)  .
\end{align*}

\begin{remark}
\label{RemSing}The Hankel function $h_{0}^{(1)}(r)$ has a singularity of order
$O(r^{-1})$ and its derivative behaves as $O\left(  r^{-2}\right)  $ for
$r\rightarrow0$ (cf. \cite{NIST:DLMF}). Note, that $A_{1}=0$ (by 
\eqref{Amn1}) and this singularity is not critical for the first interval
$\tau_{1}$. A refined estimation of the coefficient $A_{\ell}$ will be
discussed in Section \ref{subsec:refinedestimate} in order to ensure that the
\textquotedblleft near-singularities\textquotedblright\ of the Hankel function
on $\tau_{2}$, $\tau_{3}$,\ldots are damped by the smallness of the
corresponding coefficients $A_{\ell}$.
\end{remark}

\begin{notation}
We write $A\lesssim B$ if there exits a constant $C\in\mathbb{R}_{>0}$ such
that $A\leq CB$. We note that in the following estimates the constant $C$ is
independent of the boundary data $\hat{g}_{0}$, the wave speed $c$, the
frequency $\omega$, and the number of jumps $n$ but depends on $c_{\min}$
and $c_{\max}$.
\end{notation}

We present the main stability result of this paper.

\begin{theorem}
\label{thm:finalStab} Let $d=3$ and let $u$ be the solution of
\eqref{modelproblemDtN} with boundary data given by $\hat{g}(\boldsymbol{\xi
})=\hat{g}_{0}Y_{0,\mathbf{0}}(\boldsymbol{\xi})$. Let the wave speed $c$ be
constant on concentric annular regions and oscillate between the two values
$0<c_{\min}\leq c_{1},c_{2}\leq c_{\max}<\infty$, i.e.
\[
c_{j}=%
\begin{cases}
c_{1}, & \text{if }j\text{ is odd,}\\
c_{2} & \text{if }j\text{ is even,}%
\end{cases}
\qquad\forall1\leq j\leq n+1.
\]
Then it holds
\[
\left\Vert u\right\Vert _{\mathcal{H}}\lesssim C|\hat{g}_{0}|\alpha^{\omega}%
\]
for constants $\alpha>1$ which depends only on $c_{\min},c_{\max}$. For
$r\in (0,1)$, the pointwise estimates hold%
\begin{align*}
\left\vert \hat{u}\left(  r,\boldsymbol{\xi}\right)  \right\vert  &  \lesssim\alpha^{\omega}\left\vert \hat{g}%
_{0}\right\vert ,\\
\left\vert \left(  \nabla u\right)  \circ\psi\left(  r,\boldsymbol{\xi}\right)  \right\vert  &  \lesssim r\alpha^{\omega}\left\vert \hat{g}%
_{0}\right\vert .
\end{align*}

\end{theorem}

\begin{remark}
We see that for the case $d=3$ and $m=0$ the function $u$ is not only bounded
with respect to the energy norm but also pointwise. However, the bound can
grow exponentially with respect to the wave number.
\end{remark}

\section{Sharpness of the Stability Theorem and Examples of Wave Localisation}

\label{sec:examples} The representation of the entries of the Green's operator
in Theorem \ref{th:final_rep} allows us to construct examples with specific
behaviours. In this chapter we discuss two localisation phenomena for $d=3$;
for wave localisations in one dimension we refer to \cite{SauterTorres2018}.

\subsection{Localisation at a Single Discontinuity (\textquotedblleft Whispering Gallery Modes\textquotedblright)}

\label{subsec:whisperinggallerymodes} For only one discontinuity, the Green's
operator is given by a $2\times2$-matrix and the solution of the system can be
computed explicitly. Indeed one computes with the ansatz \eqref{ansatzontauj}
for $d=3$
\begin{align*}
A_{m,1}  &  =0, & A_{m,2}  &  =\frac{\operatorname{i}\omega}{c_{2}}h_{m}%
^{(1)}\left(  \frac{\omega}{c_{2}}\right)  \frac{w_{m,2,1,1}^{2,2}%
}{w_{m,2,1,1}^{1,2}},\\
B_{m,1}  &  =\frac{1}{x_{1}^{2}\omega}\frac{h_{m}^{(1)}\left(  \frac{\omega
}{c_{2}}\right)  }{w_{m,2,1,1}^{1,2}}, & B_{m,2}  &  =\frac{\operatorname{i}%
\omega}{c_{2}}h_{m}^{(1)}\left(  \frac{\omega}{c_{2}}\right)  \hat{g}_{m}.
\end{align*}
This is well-defined (cf. \cite[Lem. 1]{HansenPoignardVogelius2007}). For
fixed $c_{2}>c_{1}>0$ and $x_{1}\in\left(  0,1\right)  $, consider
$w_{m,2,1,1}^{1,2}=W_{m}(\omega)$ as a function, depending on a
\textit{complex} frequency $\omega$. Let $\tilde{\omega}_{m}^{0}$ be the
smallest (in modulus) complex zero of $W_{m}(\omega)$, and choose the
frequency $\omega_{m}=\operatorname{Re}\left(  \tilde{\omega}_{m}^{0}\right)
$. Then it can be shown that the corresponding $B_{m,1}$ has an
super-algebraic growth in $m$ (cf. \cite{CapLeadPark2012}). This well known
phenomenon is referred to \textquotedblleft whispering gallery
modes\textquotedblright{} in the literature (see e.g. \cite{CapLeadPark2012}).
The wave localises along the interface of the jump of the wave speed.

\subsection{Localisation for Highly Varying Wave Speed\label{subse:criticalex}}

In Section \ref{Ch:Repm0}, we have seen that the representation of the Green's
operator is given by \eqref{eq:repMGreen} with the simplified recursion
for $\beta_{0,\ell}$ as in \eqref{eq:recursion2m0}. In this section, we describe
parameter configurations where localisation of waves occurs.

\begin{lemma}
Let $d=3$. Let $\omega>0$ and let the wave speed \(c\) be a piecewise constant function as described in \eqref{def:wavespeed} such that
\begin{equation}
\operatorname{e}^{-\operatorname*{i}\omega\frac{h_{\ell}}{c_{\ell}} 
}\in \{\I,-\I\}\quad\forall1\leq\ell\leq n+1\quad\text{with\quad
}h_{\ell}=x_{\ell}-x_{\ell-1}. \label{eq:AssPhaseFactorCritical} 
\end{equation}
Then, the last row of the radial Green's operator satisfies
\begin{align}
\left\vert \left(  \mathbf{M}_{0}^{\operatorname{Green}}\right)  _{2\ell
-1,2n}\right\vert &=\prod_{k=\ell}^{n}\left(  \frac{1+q_{k}}{1-(-1)^{k-1}q_{k}}\right)
\label{eq:Greensoscillodd}\\ 
\left\vert \left(  \mathbf{M}_{0}^{\operatorname{Green}}\right)  _{2\ell,2n}\right\vert 
&=\left|  \Img\left(  \operatorname{e}^{\operatorname*{i}
\frac{z_{\ell}}{c_{\ell+1}}}\I^{\ell}\right)  \right|
 \prod_{k=\ell}^{n}\left(  \frac{1+q_{k}}{1-(-1)^{k-1}q_{k}}\right)
\end{align}
for \(1\leq \ell\leq n\) with the relative jumps \( q_k\) defined by
\[
q_k:=\frac{c_{k+1}-c_{k}}{c_{k+1}+c_{k}}\in\left(  -1,1\right).
\]
\end{lemma}%

\begin{proof}
The claim follows directly by using the representation \eqref{eq:repMGreen} and that under the assumption \eqref{eq:AssPhaseFactorCritical} the recursion \eqref{eq:recursion2m0} for $\left\vert \beta_{0,\ell}\right\vert$ simplifies to
\[
\left\vert\beta_{0,\ell}\right\vert=\prod_{k=1}^{\ell}\left(  \frac
{1-(-1)^{k-1}q_{k}}{1+q_{k}}\right)  ,\qquad\forall0\leq\ell\leq n.
\]
\end{proof}

If we assume that the jumps \(q_k\) are such that \(q_\ell  = (-1)^{\ell-1} \left\vert q_\ell\right\vert\) and  \(\displaystyle\min_{1\leq \ell\leq n}  \left\vert q_\ell\right\vert  \geq q_{\min}>0\), then we can derive from the \eqref{eq:Greensoscillodd} 
\begin{align}\label{eq:lowEstMGreen}
\begin{split}
\left\vert \left(  \mathbf{M}_{0}^{\operatorname{Green}}\right)  _{2\ell
-1,2n}\right\vert &=\prod_{k=\ell}^{n}\left(  \frac{1+(-1)^{k-1} \left\vert q_k\right\vert}{1- \left\vert q_k\right\vert}\right)\\
&= \prod_{\substack{k=\ell\\k \text{ odd}}}^{n}\left(  \frac{1+ \left\vert q_k\right\vert}{1-  \left\vert q_k\right\vert}\right)\geq \left(  \frac{1+q_{\min}}{1-q_{\min}}\right) ^{\left(  \lceil\frac{n}{2}\rceil-\lfloor
\frac{\ell}{2}\rfloor\right)  },
\end{split}
\end{align}
i.e. the odd entries of the Green's operator grow exponentially in the number
of jumps. Numerical experiment show that in this case the solution \(u\) localizes in the center of the domain \(\Omega = B_1^3\). 

\begin{definition}
[localisation interference]\label{DefLocInt}Let $d=3$. For any number
$n\in\mathbb{N}$ of jumps, any set of radial jump points%
\begin{equation}
0=x_{0}<x_{1}<\ldots<x_{n+1}=1, \label{defjumppoints2}%
\end{equation}
any piecewise constant wave speed, i.e.
\begin{align*}
c\left(  \mathbf{x}\right)  :=
c_{\ell}\quad \forall\mathbf{x}\in B_{1}^{d}\text{ with }\left\vert \mathbf{x}%
\right\vert \in\left(  x_{\ell-1},x_{\ell}\right)  \quad1\leq\ell\leq n+1, \label{def:cosc}
\end{align*}
that is oscillatory with \(q_1>0\), i.e.
\[q_\ell  = (-1)^{\ell-1} \left\vert q_\ell\right\vert, \quad \text{and} \quad \min_{1\leq \ell\leq n} \left\vert q_\ell\right\vert \geq q_{\min}>0\]
and any frequency $\omega>0$, we say that the wave speed $c$ and the frequency
$\omega$ are in \emph{localisation interference} if%
\begin{equation}
\mathrm{\operatorname*{e}}^{-\operatorname*{i}\omega\frac{h_{\ell}}{c_{\ell}}%
}\in \{\I,-\I\}\quad\text{with } \quad h_\ell = x_{\ell}-x_{\ell-1},\quad \forall1\leq\ell\leq n+1.
\end{equation}
\end{definition}

\begin{lemma}\label{lem:criticalconstr}
Let \(0<q\leq\frac{c_{\max}-c_{\min}}{c_{\max}+c_{\min}}\) be fixed. For any \(n\in \mathbb{N}\), there exists a frequency \(\omega \) and a wave speed \(c\) such that they are in localisation interference and \(n\sim \omega\). As a consequence, it holds
for $1\leq\ell\leq n$
\begin{align*}
\left\vert \left(  \mathbf{M}_{0}^{\operatorname{Green}}\right)  _{2\ell
-1,2n}\right\vert =\left(  \frac{1+q}{1-q}\right)  ^{\left(  \lceil\frac{n}{2}\rceil-\lfloor
\frac{\ell}{2}\rfloor\right)  },
\end{align*}
i.e. the odd entries of the Green's operator grow exponentially in the number
of jumps \(n\sim \omega \).
\end{lemma}
\begin{proof}
Choose \(0< c_1 < c_2<\infty\) such that \(\frac{c_2-c_1}{c_2+c_1}= q\). For any \(n\in \mathbb{N}\) fixed we set
\begin{align*}
c_j &= \begin{cases}
c_1 & \text{if } j \text{ is odd,}\\
c_2 & \text{if } j \text{ is even.}
\end{cases} \\
h_{\ell} &=\frac{\pi
}{2\omega}c_{\ell},\quad\text{ for }1\leq\ell\leq n+1,\\
x_{\ell}  &= \sum_{j=1}^\ell h_{\ell}, \quad 0\leq \ell \leq n+1
\end{align*}
and define the wave speed \(c\) by
\begin{align*}
c\left(  \mathbf{x}\right)  :=
c_{\ell}\quad \forall\mathbf{x}\in B_{1}^{d}\text{ with }\left\vert \mathbf{x}%
\right\vert \in\left(  x_{\ell-1},x_{\ell}\right) .
\end{align*}
In order to respect the condition $\sum_{j=1}^{n+1}h_{j}=1$, we define the frequency by
\begin{align*}
\omega := \frac{\pi}{2}\sum_{j=1}^{n+1}c_j \sim O(n).
\end{align*}
The representation of the odd entries of the Green's operator follow from \eqref{eq:lowEstMGreen}.
\end{proof}

We have seen that for any given $0<c_{1}<c_{2}$, a set of jump points can be
chosen along a frequency $\omega$ such that localisation appears. Figure \ref{fig:critical} shows examples of the solution $u$ of the Helmholtz equation with respect to the wave speed \(c\) and \(\omega\) as constructed in Lemma \ref{lem:criticalconstr} for different values of \(n\).
\begin{figure}[th]
\centerline{
\begin{subfigure}[c]{0.5\textwidth}
\includegraphics[scale=0.4]{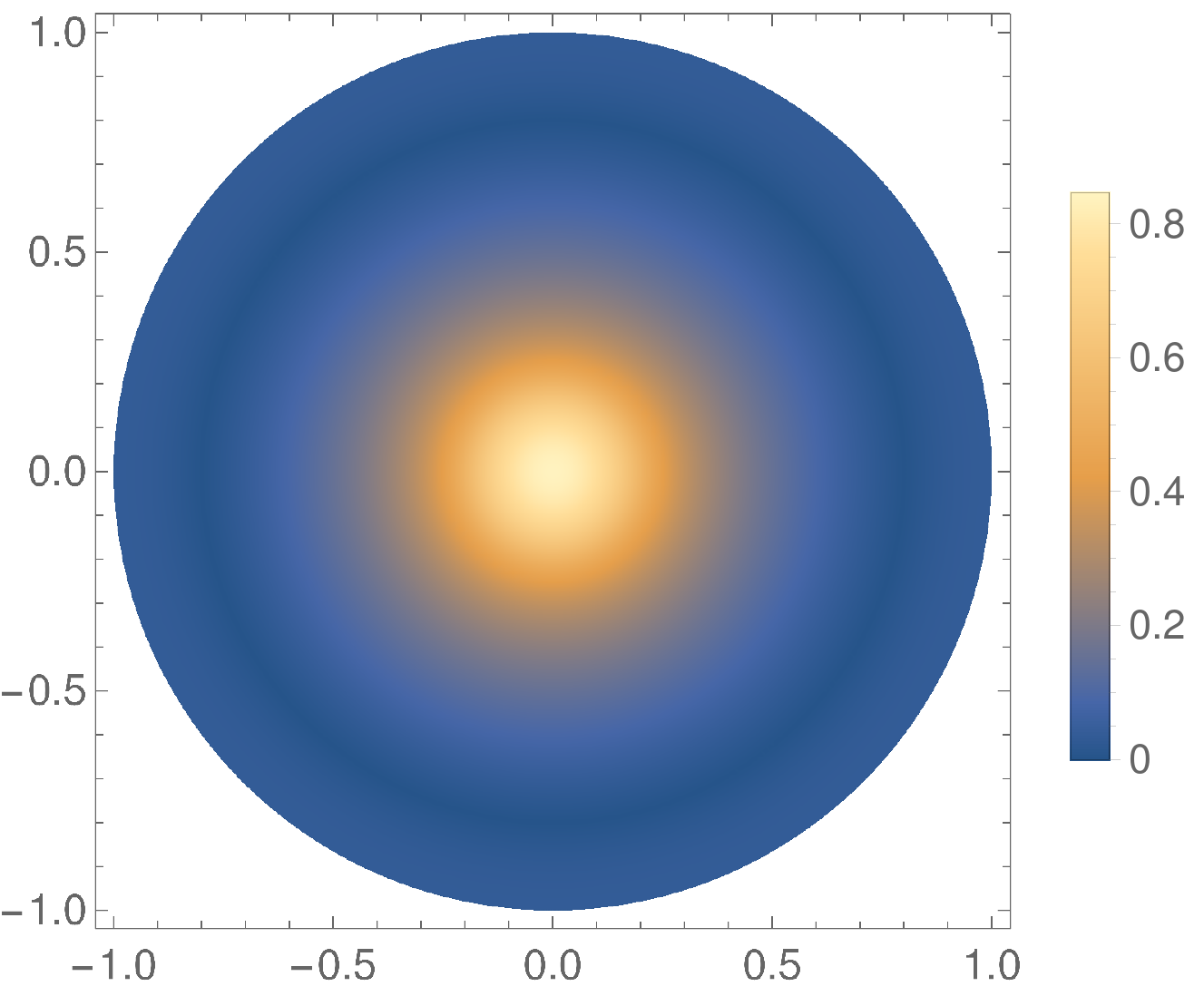}
\subcaption{\(n=2,~\omega \approx 7.85 \), \(\|\hat{u}_{0,\textbf{0}}Y_{0,\textbf{0}}\|_{\infty} \approx 0.85\)}
\end{subfigure}
\begin{subfigure}[c]{0.5\textwidth}
\includegraphics[scale=0.4]{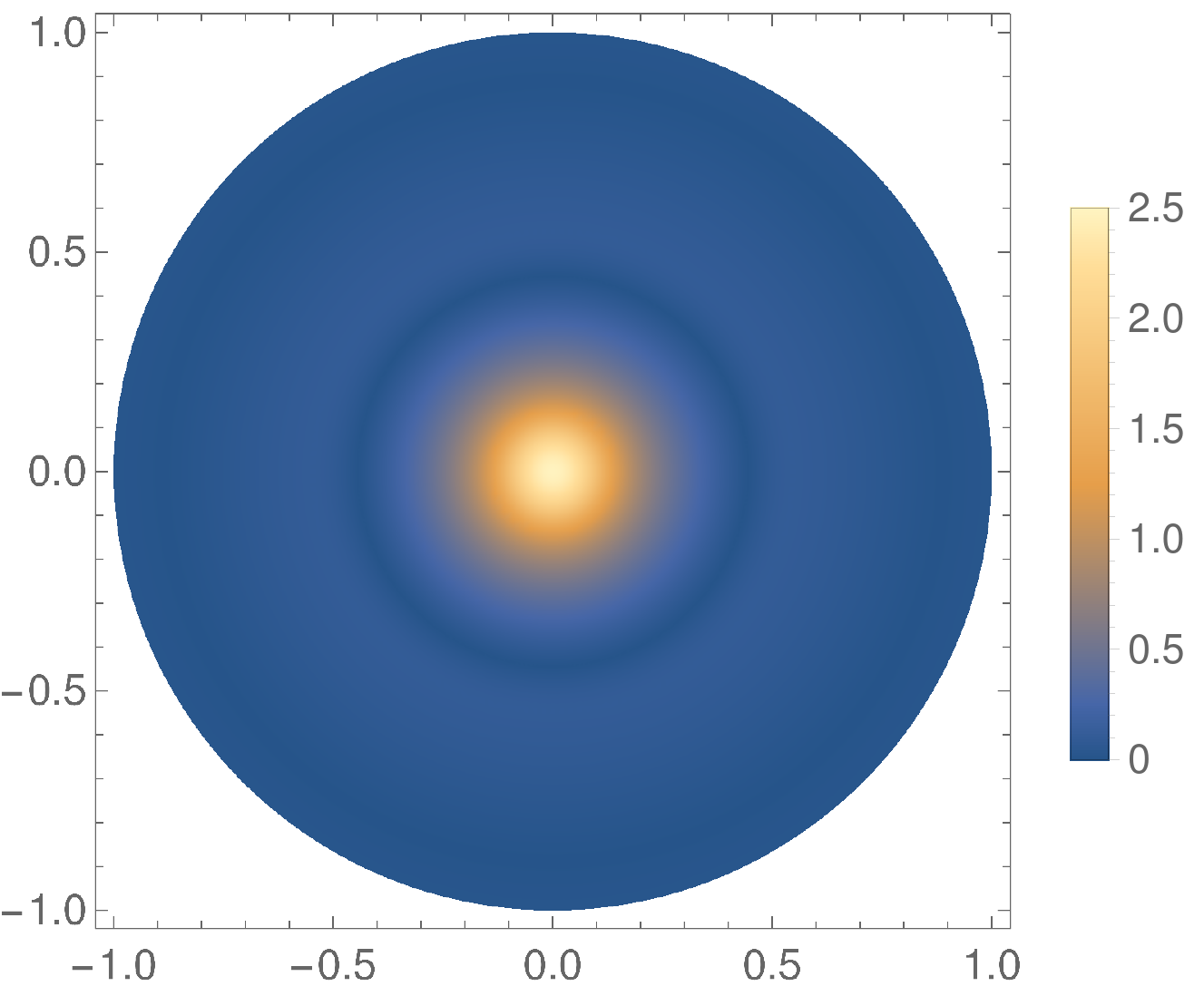}
\subcaption{\(n=4,~\omega \approx 14.1 \), \(\|\hat{u}_{0,\textbf{0}}Y_{0,\textbf{0}}\|_{\infty}\approx 2.5\)}
\end{subfigure}} \centerline{
\begin{subfigure}{0.5\textwidth}
\includegraphics[scale=0.4]{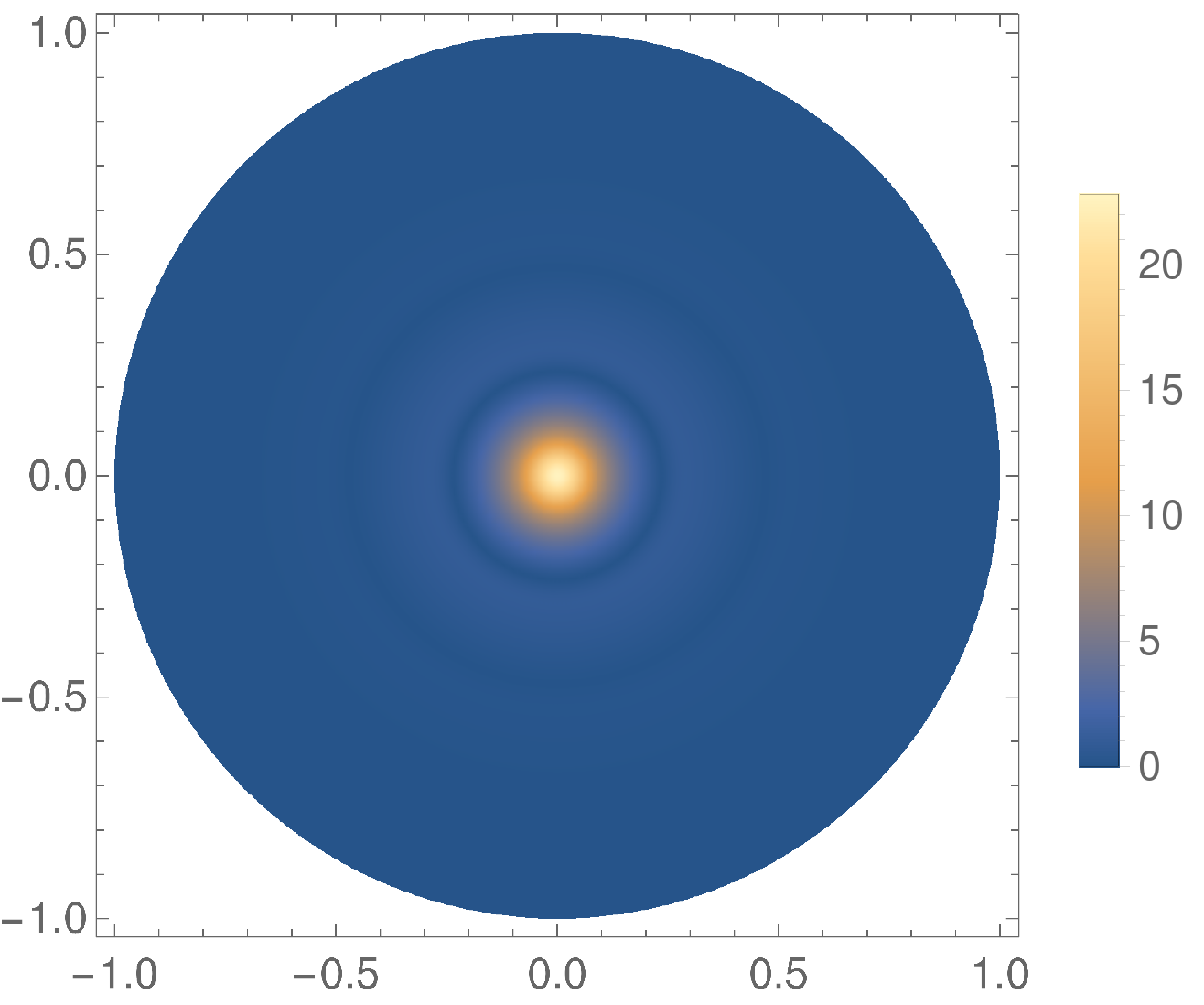}
\subcaption{\(n=8,~\omega \approx 26.7 \), \(\|\hat{u}_{0,\textbf{0}}Y_{0,\textbf{0}}\|_{\infty}\approx 22\)}
\end{subfigure}
\begin{subfigure}{0.5\textwidth}
\includegraphics[scale=0.4]{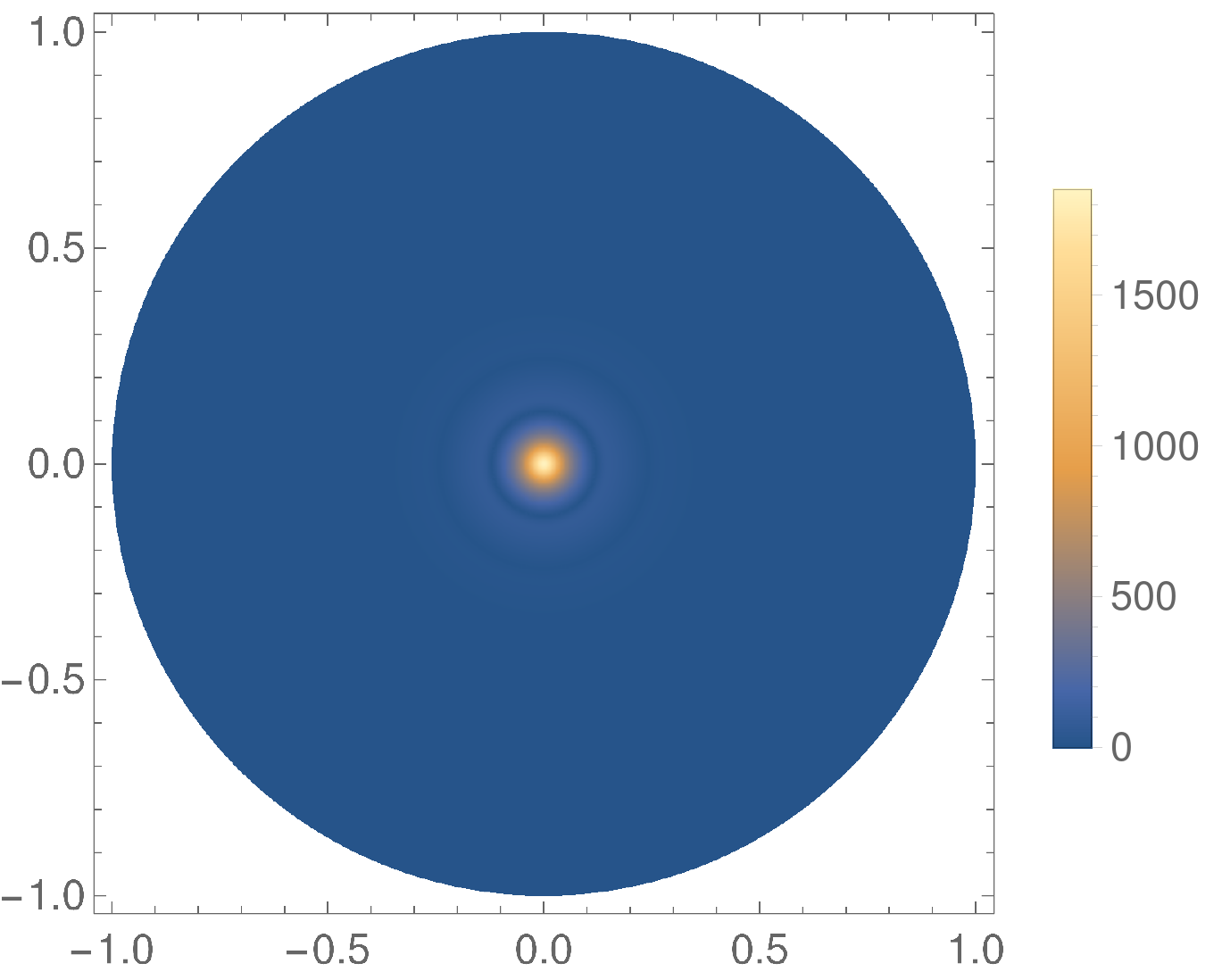}
\subcaption{\(n=16,~\omega \approx 51.8 \), \(\|\hat{u}_{0,\textbf{0}}Y_{0,\textbf{0}}\|_{\infty}\approx 1850\)}
\label{subfig:critical14}
\end{subfigure}
}\caption{Examples for a localisation with $m=0$ (i.e. the solution is radial)
and $n$ jumps. The absolute value $|\hat{u}_{0,\mathbf{0}}(r)Y_{0,\mathbf{0}%
}(\boldsymbol\xi)|$ on the disc $B_1^3\cap\{x_{3} = 0\}$ is shown. The maximal peaks
grow exponentially in $n=O (\omega)$. For any \(n\in\mathbb{N}\), the wave speed $c$ and the frequency
$\omega$ are chosen as constructed in Lemma \ref{lem:criticalconstr} with \(c_1=1,~c_2=3\). For the boundary data we set \(\hat{g}_{(0,\mathbf{0})}=1\).}%
\label{fig:critical}%
\end{figure}

\begin{remark}
A closer examination of the recursion \eqref{eq:recursion2m0} for \(\beta_{0,\ell}\) shows that there are other configurations of \(c\) and choices of \(\omega\) that do not fit into Definition \ref{DefLocInt}, but nevertheless the corresponding solution exhibits a localisation effect. In particular if \begin{equation}
\mathrm{e}^{-\operatorname*{i}\omega\frac{h_{j}}{c_{j}}}\in \{-1,1,\I,-\I\} ,\qquad\forall1\leq
j\leq n,
\end{equation}
the recursion simplifies and the representation of the last row of the Green's operator differs only by some change of signs from \eqref{eq:Greensoscillodd}. As an example, if the phase factors are such that
\begin{align*}
\operatorname{e}^{-\operatorname*{i}\omega\frac{h_{1}}{c_{1}}}&\in \{1,-1\},\\
\operatorname{e}^{-\operatorname*{i}\omega\frac{h_{\ell}}{c_{\ell}}}&\in \{\I,-\I\}\quad\forall 2\leq\ell\leq n+1,
\end{align*}
a lower estimate of the type \eqref{eq:lowEstMGreen} still holds, if the wave speed \(c\) satisfies \(q_\ell  = (-1)^{\ell}  \left\vert q_\ell\right\vert\) (instead of \(q_\ell  = (-1)^{\ell-1}  \left\vert q_\ell\right\vert\)). Other examples may be explored in a similar way. 
\end{remark}

From Theorem \ref{thm:finalStab}, we conclue that \(\|u\|_{\mathcal{H}}\) may grow exponentially in \(\omega\). Indeed for given \(n\), the solution \(u\) of the problem \eqref{eq:contHelmholtz} with wave speed \(c\), frequency \(\omega\) as constructed in Lemma \ref{lem:criticalconstr} and boundary data $\hat{g}(\boldsymbol\xi) = \hat{g}_{(0,\mathbf{0})}Y_{(0,\mathbf{0})}(\boldsymbol{\xi})$ satisfies
\begin{align*}
\|u\|_{\mathcal{H}}^2 
&\geq \left(\frac{\omega}{c_1}\right)^2\left\vert B_1\right\vert^2 \left\Vert \cos\left(\frac{\omega}{c_1}\cdot \right) 
\mathfrak{g}_{0}\right\Vert _{L^{2}\left( \tau_1\right)  }^{2}\\
&\geq \left(\frac{\omega}{c_1}\right)^2\left\vert B_1\right\vert^2 
\left\Vert \cos\left(\frac{\omega}{c_1}\cdot \right) 
\mathfrak{g}_{0}\right\Vert _{L^{2}\left( \tau_1\right)  }^{2}\\
&= \left(\frac{\omega}{c_1}\right)^2\left\vert B_1\right\vert^2 
\left( \frac{1}{4}x_1^2 + \frac{x_1}{4}\frac{c_1}{\omega}\sin\left(2\frac{\omega x_1}{c_1}\right) + 
\left (\frac{c_1}{\omega}\right)^2 \frac{\cos\left(2\frac{\omega x_1}{c_1}\right)-1}{8}\right).
\end{align*}
The choice \(x_1 = h_1 = \frac{\pi}{2} \frac{c_1}{\omega}\) implies \(\sin\left(2\frac{\omega x_1}{c_1}\right)=0\) and \(\cos\left(2\frac{\omega x_1}{c_1}\right)=-1\). 
We recall \(A_1 =0\) and identities \eqref{Greenrep} together with \eqref{fullinteriorsysfinal}, \eqref{eq:Greensoscillodd} and \eqref{modh0j0}. Therefore, 
\begin{align*}
\|u\|_{\mathcal{H}}
&\geq  \frac{1}{2}\left\vert B_1\right\vert  
\sqrt{ \left(\frac{\pi}{2}\right)^2 - 1}= \frac{\sqrt{\pi^2-4}}{4} \left(  \frac{1+q}{1-q}\right)  ^{\lceil\frac{n}{2}\rceil}\left\vert h^{(1)}_0\left(\frac{\omega}{c_{n+1}}\right)\right\vert \frac{\omega }{c_{n+1}}\left\vert \hat{g}_{(0,\mathbf{0})}\right\vert\\
&= \frac{\sqrt{\pi^2-4}}{4}  \left(  \frac{1+q}{1-q}\right)  ^{ \lceil\frac{n}{2}\rceil }\left\vert \hat{g}_{(0,\mathbf{0})}\right\vert 
\end{align*}
is a lower bound for the energy norm of the solution of problem \eqref{eq:contHelmholtz}. Combining this with the fact that \(\omega\sim n \), implies that the estimate in Theorem \ref{thm:finalStab} is sharp with respect to the frequency \(\omega\).

\subsection{Globally Stable Solution}

\label{subse:wellex} Similarly to Sections \ref{subse:criticalex} one can
construct an example, where the entries of the Green's operator are bounded
independently of the number of jumps. Let $c$ be oscillating between two
values $0<c_{1}<c_{2}$ as chosen in the proof of Lemma \ref{lem:criticalconstr}. For fixed \(n\in \mathbb{N}\), we choose the jump
locations $x_{j}$ and the frequency $\omega$ such that
\begin{equation}
\mathrm{e}^{-\operatorname*{i}\omega\frac{h_{j}}{c_{j}}}=-1,\qquad\forall1\leq
j\leq n. \label{eq:AssPhaseFactorWell}%
\end{equation}
In order to achieve \eqref{eq:AssPhaseFactorWell} we set
\[
\omega=\pi\sum_{j=1}^{n+1}c_{j}\sim O(n),\qquad h_{j}=\frac{\pi}{\omega}%
c_{j}.
\]
Then the recursion simplifies to
\[
\beta_{0,\ell}=(-1)^{\ell},\qquad\forall1\leq\ell\leq n.
\]
Finally, this leads to
\[
\left\vert \left(  \mathbf{M}_{0}^{\operatorname{Green}}\right)  _{2\ell
-1,2n}\right\vert =1,\qquad\forall1\leq\ell\leq n,
\]
and
\[
\left\vert \left(  \mathbf{M}_{0}^{\operatorname{Green}}\right)  _{2\ell
,2n}\right\vert \leq1,\qquad\forall1\leq\ell\leq n,
\]
i.e. all entries of the Green's operator are bounded in modulus from above by
$1$.

Figure \ref{fig:wellbehaved} shows examples of solutions constructed in this way
for different values of $n$. 
\begin{figure}[th]
\begin{subfigure}{0.5\textwidth}
\includegraphics[width=\textwidth]{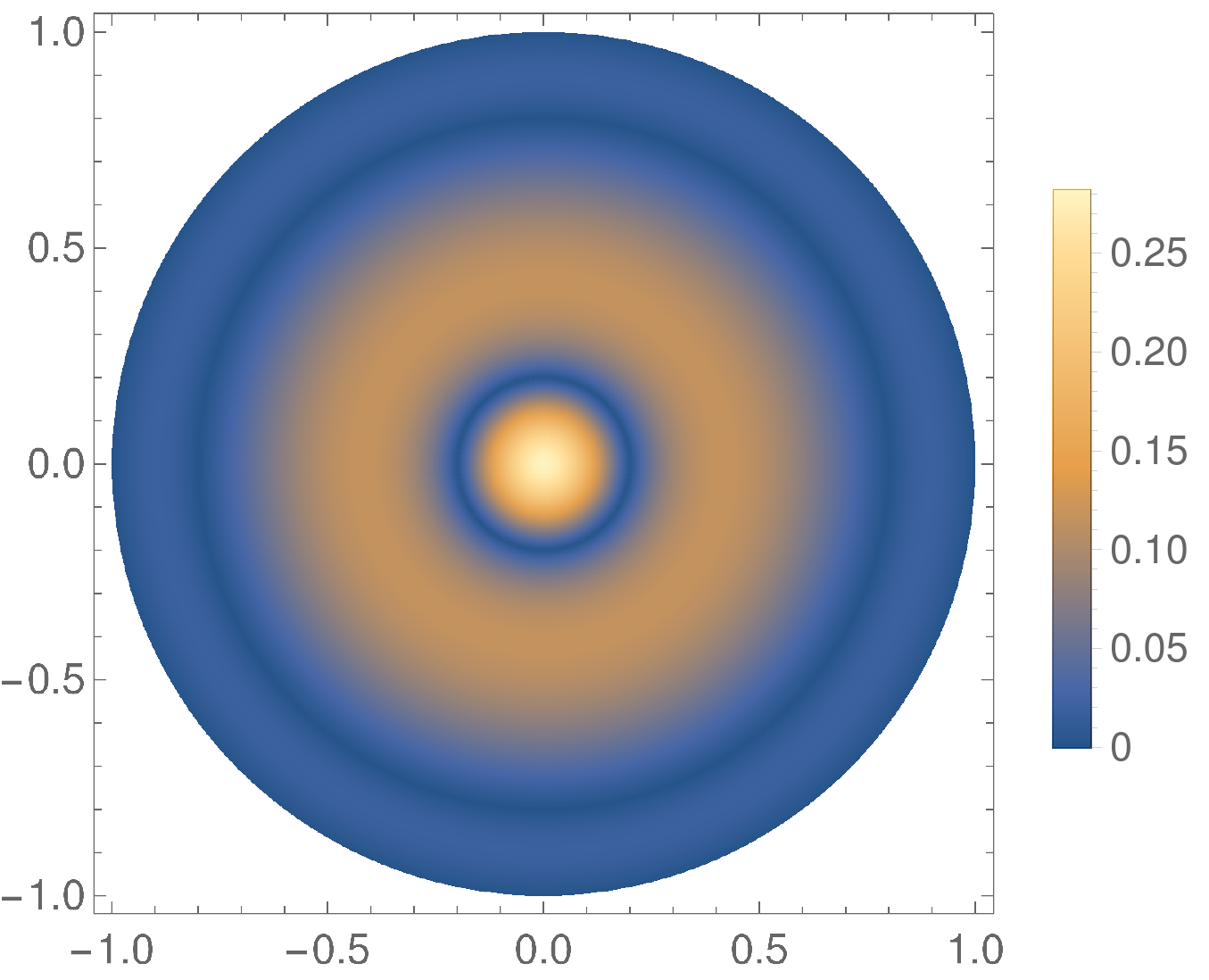}
\subcaption{\(n=2,~\omega \approx 7.85 \), \(\|\hat{u}_{0,\textbf{0}}Y_{0,\textbf{0}}\|_{\infty}\approx 0.28\)}
\end{subfigure}
\begin{subfigure}{0.5\textwidth}
\includegraphics[width=\textwidth]{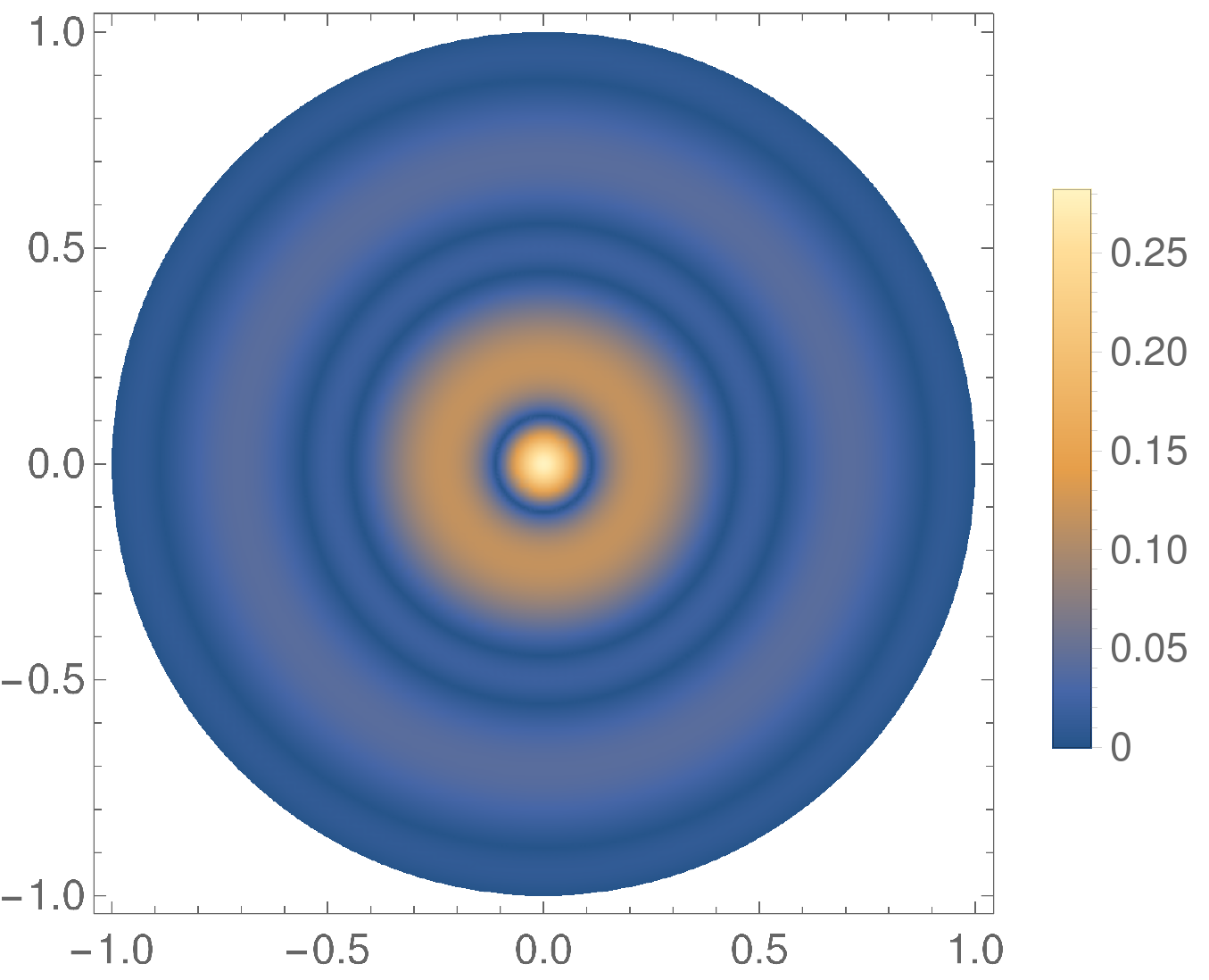}
\subcaption{\(n=4,~\omega \approx 14.1 \), \(\|\hat{u}_{0,\textbf{0}}Y_{0,\textbf{0}}\|_{\infty}\approx 0.28\)}
\end{subfigure}
\begin{subfigure}{0.5\textwidth}
\includegraphics[width=\textwidth]{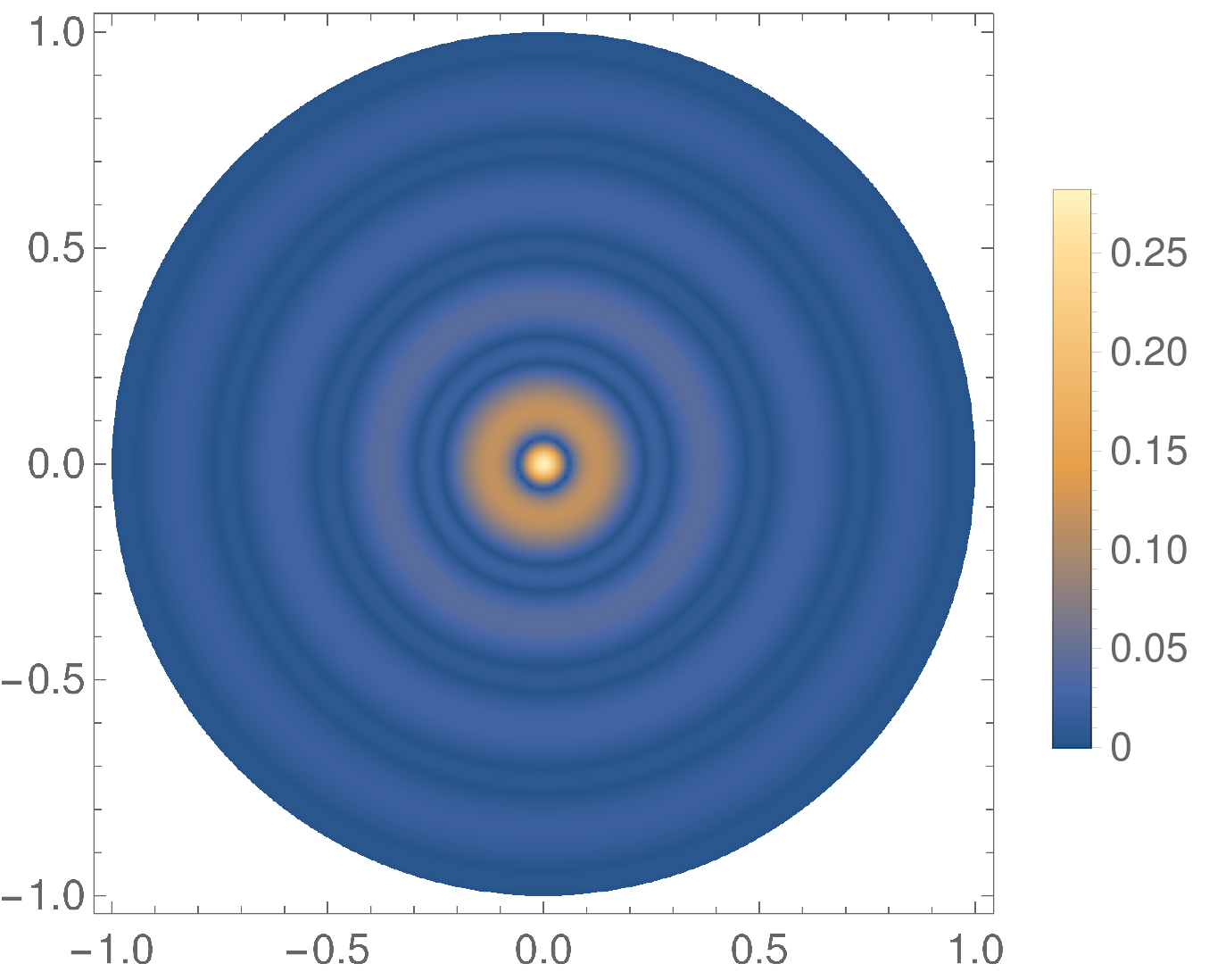}
\subcaption{\(n=18,~\omega \approx 26.7 \), \(\|\hat{u}_{0,\textbf{0}}Y_{0,\textbf{0}}\|_{\infty}\approx 0.28\)}
\end{subfigure}
\begin{subfigure}{0.5\textwidth}
\includegraphics[width=\textwidth]{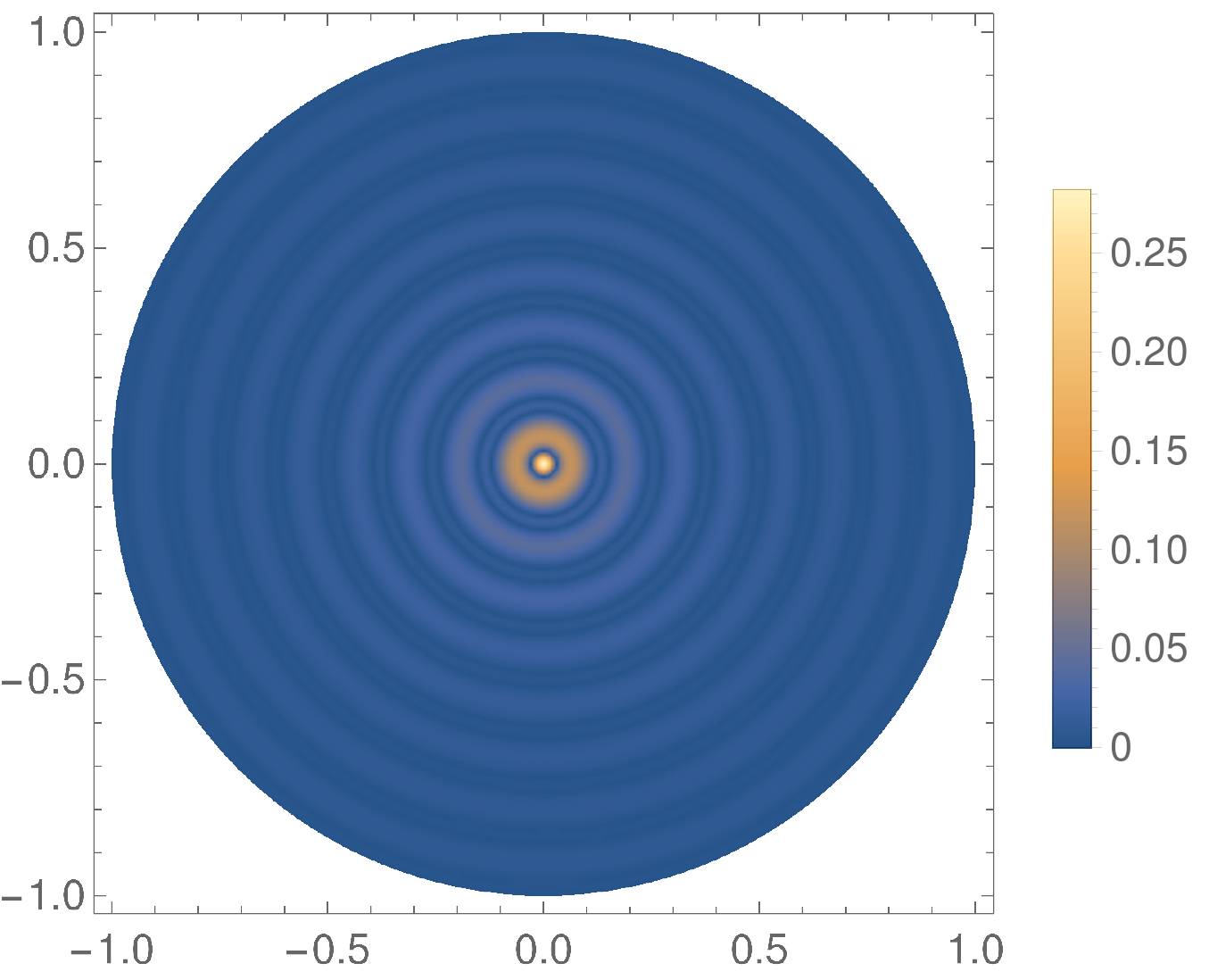}
\subcaption{\(n=16,~\omega \approx 51.8 \), \(\|\hat{u}_{0,\textbf{0}}Y_{0,\textbf{0}}\|_{\infty}\approx 0.28\)}
\label{subfig:well14}
\end{subfigure}
\caption{Examples with $m=0$ (i.e. the solution is radial) and $n$ jumps,
where the solution behaves nicely. The absolute value $|\hat{u}_{0,\mathbf{0}%
}(r)Y_{0,\mathbf{0}}(\boldsymbol\xi)|$ on the disc $B_1^3\cap\{x_{3}=0\}$ is shown. For any \(n\in\mathbb{N}\), the
wave speed $c$ and the frequency $\omega$ are chosen as constructed in Section
\ref{subse:wellex} with \(c_1=1,~c_2=3\). For the boundary data we set \(\hat{g}_{(0,\mathbf{0})}=1\).}
\label{fig:wellbehaved}
\end{figure}

\begin{remark}
Although, we have seen in Section \ref{subse:criticalex} that the stability estimate in Theorem \ref{thm:finalStab} is sharp, these examples show that the estimate can be pessimistic for some configurations of \(c\) and \(\omega\).
Moreover for fixed, large number of jumps $n$, a piecewise constant wave speed $c$
along a specific frequency $\omega$ chosen such that the wave localizes is
rare within the parameter space \eqref{defomega} together with \eqref{def:wavespeed} and can
be interpreted as an interference phenomenon. As soon as such a critical
choice is slightly perturbed (either by changing $\omega$ or $c$ (and this
includes changing a single jump point $x_{i}$)) the wave becomes typically well
behaved. The examples of Section \ref{subse:criticalex} and \ref{subse:wellex}
have identical wave speed $c$ and only differ by the frequency $\omega$.
\end{remark}

\section{Stability for Waves for the Lowest Fourier Mode $m=0$}

\label{sec:stabilitym0} In this section, we derive an estimate of the entries
of the Green's operator $\mathbf{M}_{m}^{\operatorname{Green}}$ for $m=0$,
i.e. with no angular oscillations, and $d=3$. That is, a further investigation
of the sequence $\beta_{0,\ell}$ as defined in \eqref{eq:recursion2m0} is
needed. In the first step, we will find an upper and lower bound on
$|\beta_{0,\ell}|$ which is independent of the position of the scaled jumps
$z_{\ell}=\omega x_{\ell}$ and their number $n$. In a second step, we derive a
refined estimate of $\beta_{0,\ell}$, in particular of $\operatorname{Im}%
\left(  \operatorname*{e}{}^{\operatorname*{i}\frac{z_{\ell}}{c_{\ell+1}}%
}\beta_{0,\ell}\right)  $ (cf. \eqref{eq:repMGreen}), for small values of
$z_{\ell}$. Such a refined estimate of the odd entries in the Green's
operator, i.e. the coefficients belonging to the Hankel function $h_{0}\left(
\frac{\omega x}{c_{\ell}}\right)  $ is needed, because $h_{0}^{\prime}\left(
\frac{\omega x}{c_{\ell}}\right)  $ has a singularity of order $O\left(
x^{-2}\right)  $ at $x=0$ (Lemma \ref{lem:hankelpiecewise}).

Throughout this section, we always assume that the wave speed $c$ oscillates
between two values, i.e.
\begin{equation}
q_{\ell}=\left(  -1\right)  ^{\ell+1}q\quad\text{with }-1<q=\frac{c_{2}-c_{1}%
}{c_{2}+c_{1}}<1, \label{propql}%
\end{equation}
for some $c_{\min}\leq c_{1},c_{2}\leq c_{\max}$. Also, we define
\begin{equation}
\delta_{\ell}:=\frac{\omega h_{\ell}}{c_{\ell}}. \label{defdeltaell}%
\end{equation}

\subsection{Maximal Growth of the Entries of the Green's Operator}

The goal of this section is a lower and upper estimate on the sequence
$\beta_{0,\ell}$ as stated in Proposition \ref{th:uplowbeta}.

\begin{proposition}
\label{th:uplowbeta}
It holds
\begin{align*}
\left\vert \beta_{0,\ell}\right\vert  &  \leq\alpha^{\omega},\\
\left\vert \beta_{0,\ell}\right\vert  &  \geq\alpha^{-\omega},
\end{align*}
for some $\alpha=\alpha(q)>1$ for $q\in\left(  -1,1\right)  $.
\end{proposition}

\begin{proof}
From \eqref{eq:recursion2m0} and \eqref{propql} we obtain
%
\[
\left\vert \beta_{0,\ell}\right\vert ^{2}=\dfrac{1+2q_{\ell}\operatorname{Re}%
\left(  \operatorname{e}^{-2\operatorname*{i}\delta_{\ell}}\frac{\beta
_{0,\ell-1}^{2}}{\left\vert \beta_{0,\ell-1}\right\vert ^{2}}\right)
+q_{\ell}^{2}}{\left(  1+q_{\ell}\right)  ^{2}}\left\vert \beta_{0,\ell
-1}\right\vert ^{2}.
\]
This leads to%
\begin{equation}
\frac{1-\left\vert q_{\ell}\right\vert }{1+q_{\ell}}\left\vert \beta
_{0,\ell-1}\right\vert \leq\left\vert \beta_{0,\ell}\right\vert \leq
\frac{1+\left\vert q_{\ell}\right\vert }{1+q_{\ell}}\left\vert \beta
_{0,\ell-1}\right\vert . \label{onerec}%
\end{equation}
Next, we express $\beta_{0,\ell}$ in terms of $\beta_{0,\ell-2}$ for $\ell
\geq2$:%
\begin{equation}
\beta_{0,\ell}:=\frac{\operatorname{e}^{-\operatorname*{i}\left(  \delta
_{\ell-1}+\delta_{\ell}\right)  }}{1-q^{2}}\left(  \left(  1-q^{2}%
\operatorname{e}^{2\operatorname*{i}\delta_{\ell}}\right)  \beta_{0,\ell
-2}+\left(  -1\right)  ^{\ell}q\left(  1-\operatorname{e}^{2\operatorname*{i}%
\delta_{\ell}}\right)  \operatorname{e}^{2\operatorname*{i}\delta_{\ell-1}%
}\overline{\beta_{0,\ell-2}}\right)  \label{doublerec1}%
\end{equation}
and for the square of the modulus:%
\begin{align*}
\left\vert \beta_{0,\ell}\right\vert ^{2}=  &  \frac{\left\vert \beta
_{0,\ell-2}\right\vert ^{2}}{\left(  1-q^{2}\right)  ^{2}}\left(  \left(
1+2q^{2}+q^{4}-4q^{2}\operatorname{Re}\operatorname{e}^{2\operatorname*{i}%
\delta_{\ell}}\right)  \right. \\
&  \left.  +2\left(  -1\right)  ^{\ell}q\operatorname{Re}\left(  \left(
1-q^{2}\operatorname{e}^{2\operatorname*{i}\delta_{\ell}}\right)  \left(
1-\operatorname{e}^{-2\operatorname*{i}\delta_{\ell}}\right)  \operatorname{e}%
^{-2\operatorname*{i}\delta_{\ell-1}}\frac{\beta_{0,\ell-2}^{2}}{\left\vert
\beta_{0,\ell-2}\right\vert ^{2}}\right)  \right) \\
=  &  \frac{\left\vert \beta_{0,\ell-2}\right\vert ^{2}}{\left(
1-q^{2}\right)  ^{2}}\left(  \left(  1+2q^{2}+q^{4}-4q^{2}\operatorname{Re}%
\operatorname{e}^{2\operatorname*{i}\delta_{\ell}}\right)  \right. \\
&  \left.  +2\left(  -1\right)  ^{\ell}q\left(  \left(  1+q^{2}\right)
\left(  1-\operatorname{Re}\operatorname{e}^{2\operatorname*{i}\delta_{\ell}%
}\right)  \right)  \operatorname{Re}\left(  \operatorname{e}%
^{-2\operatorname*{i}\delta_{\ell-1}}\frac{\beta_{0,\ell-2}^{2}}{\left\vert
\beta_{0,\ell-2}\right\vert ^{2}}\right)  \right) \\
&  \left.  +2\left(  -1\right)  ^{\ell}q\left(  1-q^{2}\right)  \left(
\operatorname{Im}\operatorname{e}^{2\operatorname*{i}\delta_{\ell}}\right)
\operatorname{Im}\left(  \operatorname{e}^{-2\operatorname*{i}\delta_{\ell-1}%
}\frac{\beta_{0,\ell-2}^{2}}{\left\vert \beta_{0,\ell-2}\right\vert ^{2}%
}\right)  \right)
\end{align*}
Note that for $\delta_{\ell}\geq0$, it holds
\begin{align*}
-\operatorname{Re}\operatorname{e}^{2\operatorname*{i}\delta_{\ell}}  &
=-\cos(2\delta_{\ell})\leq-1+2\min\left\{  \delta_{\ell}^{2},1\right\}  ,\\
\operatorname{Im}\operatorname{e}^{2\operatorname*{i}\delta_{\ell}}  &
=\sin(2\delta_{\ell})\leq\min\left\{  2\delta_{\ell},1\right\}
\end{align*}
so that%
\begin{align*}
\left\vert \beta_{0,\ell}\right\vert ^{2}\leq &  \frac{\left\vert
\beta_{0,\ell-2}\right\vert ^{2}}{\left(  1-q^{2}\right)  ^{2}}\Bigg((1-q^{2}%
)^{2}+8q^{2}\min\left\{  \delta_{\ell}^{2},1\right\}  +4|q|\left(
1+q^{2}\right)  \min\left\{  \delta_{\ell}^{2},1\right\} \\
+  &  2|q|\left(  1-q^{2}\right)  \min\left\{  2\delta_{\ell},1\right\}
\Bigg)\\
\leq &  \left\vert \beta_{0,\ell-2}\right\vert ^{2}\left(  1+\frac{C_{0}%
|q|}{(1-q^{2})^{2}}\min\left\{  \delta_{\ell},1\right\}  \right)
\end{align*}
for $C_{0}=20$. We consider the majorant
\[
\left\vert \beta_{0,2\ell}\right\vert ^{2}\leq\rho_{\ell}^{2}((\delta
_{2k})_{k=1}^{\ell}):=\prod_{k=1}^{\ell}\left(  1+\frac{C_{0}|q|}%
{(1-q^{2})^{2}}\min\left\{  \delta_{2k},1\right\}  \right)  .
\]
For a fixed $0<\gamma<1$, define the two disjoint subsets of indices
\begin{align*}
I_{h}  &  =\left\{  j\in{1,...,\ell}~|~\delta_{j}\leq\gamma\right\} \\
I_{H}  &  =\left\{  j\in{1,...,\ell}~|~\delta_{j}>\gamma\right\}  .
\end{align*}
Then we can write (with $\left\vert I_{H}\right\vert $ denoting the
cardinality of $I_{H}$)%
\[
\prod_{k=1}^{\ell}\left(  1+\frac{C_{0}|q|}{(1-q^{2})^{2}}\min\left\{
\delta_{2k},1\right\}  \right)  =\left(  1+\frac{C_{0}|q|}{(1-q^{2})^{2}%
}\right)  ^{|I_{H}|}\prod_{k\in I_{h}}\left(  1+\frac{C_{0}|q|}{(1-q^{2})^{2}%
}\delta_{2k}\right)
\]
and treat the two factors differently. The second factor can be treated with
the techniques used in the proof of \cite[Proposition 19]{SauterTorres2018}.
That is, we maximize
\begin{align}\label{eq:factor2}
\prod_{k\in I_{h}}\left(  1+\frac{C_{0}|q|}{(1-q^{2})^{2}}\delta_{2k}\right)
\end{align}
under the restriction $\sum_{k\in I_{h}}\delta_{2k}\leq s$. The restriction
comes from the fact that
\[
\sum_{k\in I_{h}}\delta_{2k}\leq\sum_{k=1}^{2\ell}2\omega\frac{h_{k}}{c_{k}%
}\leq2\frac{\omega}{c_{\min}}\sum_{k=1}^{2\ell}h_{k}\leq2\frac{\omega}%
{c_{\min}}=:s.
\]
The same analysis as in the proof of \cite[Proposition 19]{SauterTorres2018},
shows that there are three possible candidates for the maximizers of \eqref{eq:factor2}:

\begin{enumerate}
[(a)]

\item if $|I_{h}| \leq\frac{s}{\gamma}$: Let $\delta_{2k} = \gamma$ for all
$k\in I_{h}$. Then we have
\begin{align*}
\prod_{k\in I_{h}}\left(  1+\frac{C_{0}|q|}{(1-q^{2})^{2}}\delta_{2k}\right)
&  = \left(  1+\frac{C_{0}|q|}{(1-q^{2})^{2}}\gamma\right)  ^{|I_{h}|}
\leq\left(  \left(  1+\frac{C_{0}|q|}{(1-q^{2})^{2}}\gamma\right)  ^{\frac
{1}{\gamma}}\right)  ^{s}\\
&  \leq\exp\left(  \frac{C_{0}|q|}{(1-q^{2})^{2}} \right)  ^{s}%
\end{align*}

\item if $|I_{h}|>\frac{s}{\gamma}$: Let
\begin{align*}
\delta_{2k_{j}}  &  =\gamma\qquad\qquad\qquad\text{for }1\leq j\leq L-1,\\
\delta_{2k_{L}}  &  =s-(L-1)\gamma,\\
\delta_{2k_{j}}  &  =0\qquad\qquad\qquad\text{for }L+1\leq j\leq|I_{h}|,
\end{align*}
where $L=\min\{\ell\in\mathbb{N}\mid\ell\gamma\geq s\}$ and $I_{h}%
=\{k_{1},\dots,k_{|I_{h}|}\}$. Then we use the fact that $L\leq\frac{s}%
{\gamma}+1$ and compute
\begin{align*}
\prod_{k\in I_{h}}\left(  1+\frac{C_{0}|q|}{(1-q^{2})^{2}}\delta_{2k}\right)
&  =\left(  1+\frac{C_{0}|q|}{(1-q^{2})^{2}}\gamma\right)  ^{L}\leq2\left(
\left(  1+\frac{C_{0}|q|}{(1-q^{2})^{2}}\gamma\right)  ^{\frac{1}{\gamma}%
}\right)  ^{s}\\
&  \leq2\exp\left(  \frac{C_{0}|q|}{(1-q^{2})^{2}}\right)  ^{s}.
\end{align*}

\item if $|I_{h}|>\frac{s}{\gamma}$: Let $\delta_{2k}=\frac{s}{|I_{h}|}$ for
all $k\in I_{h}$. In this case
\[
\prod_{k\in I_{h}}\left(  1+\frac{C_{0}|q|}{(1-q^{2})^{2}}\delta_{2k}\right)
\leq\left(  1+\frac{C_{0}|q|}{(1-q^{2})^{2}}\frac{s}{|I_{h}|}\right)
^{|I_{h}|}\leq\exp\left(  \frac{C_{0}|q|}{(1-q^{2})^{2}}\right)  ^{s}.
\]

\end{enumerate}

Since $s\leq2\frac{\omega}{c_{\min}}$ all choices lead to the estimate
\[
\prod_{k\in I_{h}}\left(  1+\frac{C_{0}|q|}{(1-q^{2})^{2}}\delta_{2k}\right)
\leq\tilde{\alpha}_{q}^{\omega}%
\]
for some $\tilde{\alpha}_{q}>1$ depending on $C_{0}$. For the first factor, we
note that
\[
\delta_{j}\geq\gamma\iff h_{j}\geq\gamma\frac{c_{j}}{\omega}.
\]
This means that $|I_{H}|\leq\frac{\omega}{\gamma c_{\min}}$ and therefore%
\begin{equation}
\left(  1+\frac{C_{0}|q|}{(1-q^{2})^{2}}\right)  ^{|I_{H}|}\leq\tilde{\alpha
}_{q}^{\omega}, \label{upperestdoublerec}%
\end{equation}
for some $\tilde{\alpha}_{q}>1$ depending on $C_{0}$. The lower bound can be
treated similarly as the upper bound. First, we observe that
\[
\left\vert \beta_{0,2\ell}\right\vert ^{2}\geq\left\vert \beta_{0,2\ell
-2}\right\vert ^{2}\left(  1-\frac{C_{0}|q|}{(1-q^{2})^{2}}\min\left\{
\delta_{2\ell},1\right\}  \right)  .
\]
From \eqref{onerec} we get by using \eqref{propql}
\[
\left\vert \beta_{0,2\ell}\right\vert \geq\frac{1-|q_{\ell}|}{1+q_{\ell}%
}\left\vert \beta_{0,2\ell-1}\right\vert \geq\frac{1-|q|}{1+|q|}\left\vert
\beta_{0,2\ell-2}\right\vert >0.
\]
This leads to
\[
\left\vert \beta_{0,2\ell}\right\vert ^{2}\geq\left\vert \beta_{0,2\ell
-2}\right\vert ^{2}\max\left\{  \left(  1-\frac{C_{0}|q|}{(1-q^{2})^{2}}%
\min\left\{  \delta_{2\ell},1\right\}  \right)  ,\left(  \frac{1-|q|}%
{1+|q|}\right)  ^{2}\right\}  .
\]
We set $\min\left\{  \frac{(1-q^{2})^{2}}{C_{0}|q|},1\right\}  >\gamma>0$ and
again split
\[
\left\vert \beta_{0,2\ell}\right\vert ^{2}\geq\left(  \frac{1-|q|}%
{1+|q|}\right)  ^{4|I_{H}|}\prod_{k\in I_{h}}\left(  1-\frac{C_{0}%
|q|}{(1-q^{2})^{2}}\delta_{2k}\right)  .
\]
The estimate of the first factor follows by using the same arguments as for
\eqref{upperestdoublerec}. For the second one, we use that $\frac{C_{0}%
|q|}{(1-q^{2})^{2}}\delta_{2k}<1$ for $k\in I_{h}$ and repeat the proof for
the upper bound by the same symmetry and opposite monotonicity property.
Finally, the estimate of $\left\vert \beta_{0,2\ell-1}\right\vert $ for odd
indices follows easily by using \eqref{onerec}.
\end{proof}

\begin{remark}
We note that the constant $\alpha_{q}>1$ only depends on $\left\vert q\right\vert\in\left(
0,1\right)  $ (however, in a continuous way). As $\left\vert q\right\vert
\rightarrow1$ (which implies $c_{\max}\rightarrow\infty$ or $c_{\min
}\rightarrow0$) the constant $\alpha_{q}$ may tend to $\infty$. We emphasize
that we do not need any periodicity on the wave speed $c$, and that the
estimate is independent of the number of jumps $n$ in $c$ or the positions
$x_{\ell}$ of the jumps.
\end{remark}

\subsection{A Refined Estimation for the Coefficients of the Hankel Function
Close to the Origin}

\label{subsec:refinedestimate} In this chapter, we establish a refined
estimate of $\beta_{0,\ell}$ for small $\frac{\omega x_{\ell}}%
{c_{\ell}}$ for reasons explained in Remark \ref{RemSing}. In this section, we
only consider indices $\ell$ such that%
\begin{equation}
\frac{\omega x_{\ell}}{c_{\min}}<\frac{1}{4C_{1}}\quad\text{for\quad}%
C_{1}:=\frac{\max\{1,c_{\max}\}}{c_{\min}^{2}} \label{defC1}%
\end{equation}
(otherwise, the singular behaviour of the Hankel function at the origin is
non-critical in the sense that all constants, in general, may depend on
$c_{\min}$ and $c_{\max}$).
Let $L$ be the maximal index such that $\frac{\omega x_{\ell}}{c_{\min}%
}<\left(  4C_{1}\right)  ^{-1}$ holds. In view of \eqref{GreenFunRep}, our
goal is to find an estimate of the term
\[
\operatorname{Im}\left(  \mathrm{e}^{\operatorname*{i}\frac{z_{\ell}}%
{c_{\ell+1}}}\beta_{0,\ell}\right)  =\sin\left(  \frac{z_{\ell}}{c_{\ell+1}%
}\right)  \operatorname{Re}\left(  \beta_{0,\ell}\right)  +\cos\left(
\frac{z_{\ell}}{c_{\ell+1}}\right)  \operatorname{Im}\left(  \beta_{0,\ell
}\right)  .
\]
For $2\leq\ell\leq L$ we recall the recursion for $\beta_{0,\ell}$ (cf.
\eqref{doublerec1}):
\begin{align*}
\beta_{0,\ell}=&\mathrm{e}^{-\operatorname*{i}\frac{\omega h_{\ell-1}}%
{c_{\ell-1}}}\left(  \frac{\mathrm{e}^{-\operatorname*{i}\frac{\omega h_{\ell
}}{c_{\ell}}}-q^{2}\mathrm{e}^{\operatorname*{i}\frac{\omega h_{\ell}}%
{c_{\ell}}}}{1-q^{2}}\right)  \beta_{0,\ell-2}
\\
&-(-1)^{\ell}\frac
{2\operatorname*{i}q}{1-q^{2}}\mathrm{e}^{\operatorname*{i}\frac{\omega
h_{\ell-1}}{c_{\ell-1}}}\sin\left(  \frac{\omega h_{\ell}}{c_{\ell}}\right)
\overline{\beta_{0,\ell-2}}%
\end{align*}
and split it into the real part of $\beta_{0,\ell}^{R}:=\operatorname{Re}%
\left(  \beta_{0,\ell}\right)  $ and imaginary part $\beta_{0,\ell}%
^{I}:=\operatorname{Im}\left(  \beta_{0,\ell}\right)  $ of $\beta_{0,\ell}$ .
This leads to%
\[%
\begin{pmatrix}
\beta_{0,\ell}^{R}\\
\beta_{0,\ell}^{I}%
\end{pmatrix}
=\mathbf{P}_{\ell}%
\begin{pmatrix}
\beta_{0,\ell-2}^{R}\\
\beta_{0,\ell-2}^{I}%
\end{pmatrix}
,\qquad\mathbf{P}_{\ell}:=%
\begin{bmatrix}
1+s_{\ell}\left(  \mu_{\ell}\right)  & t_{\ell}\left(  \mu_{\ell}\right) \\
-t_{\ell}\left(  \mu_{\ell}^{-1}\right)  & 1+s_{\ell}\left(  \mu_{\ell}%
^{-1}\right)
\end{bmatrix}
\]
with $\delta_{\ell}$ as in \eqref{defdeltaell} and%
\begin{align*}
\mu &  :=\frac{1-q}{1+q}=\frac{c_{1}}{c_{2}},\quad\mu_{\ell}:=\left\{
\begin{array}
[c]{ll}%
\mu & \ell\text{ is even,}\\
\mu^{-1} & \ell\text{ is odd,}%
\end{array}
\right. \\
s_{\ell}\left(  \kappa\right)   &  :=\cos\left(  \delta_{\ell-1}\right)
\cos\left(  \delta_{\ell}\right)  -1-\kappa\sin\left(  \delta_{\ell-1}\right)
\sin\left(  \delta_{\ell}\right)  ,\\
t_{\ell}\left(  \kappa\right)   &  :=\sin\left(  \delta_{\ell-1}\right)
\cos\left(  \delta_{\ell}\right)  +\kappa\cos\left(  \delta_{\ell-1}\right)
\sin\left(  \delta_{\ell}\right)  .
\end{align*}
Since $\beta_{0,0}=1$, the recursion gives for $\beta_{0,2\ell}$ for even
indices%
\begin{equation}%
\begin{pmatrix}
\beta_{0,2\ell}^{R}\\
\beta_{0,2\ell}^{I}%
\end{pmatrix}
=\mathbf{P}_{2\ell}\cdot\ldots\cdot\mathbf{P}_{4}\cdot\mathbf{P}_{2}\left(
\begin{array}
[c]{c}%
1\\
0
\end{array}
\right)  . \label{eq:betaRealImag}%
\end{equation}
The values of $\beta_{0,k}$ for odd indices $k$ are given by%
\begin{equation}%
\begin{pmatrix}
\beta_{0,2\ell+1}^{R}\\
\beta_{0,2\ell+1}^{I}%
\end{pmatrix}
=%
\begin{bmatrix}
\cos(\delta_{2\ell+1}) & \sin(\delta_{2\ell+1})\\
-\mu\sin(\delta_{2\ell+1}) & \mu\cos(\delta_{2\ell+1})
\end{bmatrix}%
\begin{pmatrix}
\beta_{0,2\ell}^{R}\\
\beta_{0,2\ell}^{I}%
\end{pmatrix}
. \label{eq:betaRealImagOdd}%
\end{equation}
At this point, we only discuss the recursion for the even values $2\ell$ of
$\beta_{0,\cdot}~$. The result for the odd values will be proved in Lemma
\ref{lem:RepBeta}. The proof of the following Lemma is straightforward by induction.

\begin{lemma}
It holds
\begin{equation}
\mathbf{P}_{2\ell}\cdot\ldots\cdot\mathbf{P}_{2}=%
\begin{bmatrix}
1 & L_{2\ell}\left(  \mu\right) \\
-L_{2\ell}\left(  \mu^{-1}\right)  & 1
\end{bmatrix}
+%
\begin{bmatrix}
S_{2\ell}\left(  \mu\right)  & T_{2\ell}\left(  \mu\right) \\
-T_{2\ell}\left(  \mu^{-1}\right)  & S_{2\ell}\left(  \mu^{-1}\right)
\end{bmatrix}
, \label{pprodlem}%
\end{equation}
with%
\[
L_{2\ell}\left(  \kappa\right)  :=\sum_{j=1}^{\ell}t_{2j}\left(
\kappa\right)  .
\]
$S_{2\ell}\left(  \kappa\right)  ,T_{2\ell}\left(  \kappa\right)  $ are given
by the recursion
\begin{equation}
S_{2}\left(  \kappa\right)  :=s_{2}\left(  \kappa\right)  ,\qquad T_{2}\left(
\kappa\right)  :=0, \label{DefS1T1}%
\end{equation}
and%
\begin{align*}
S_{2\ell+2}\left(  \kappa\right)   &  =(1+s_{2\ell+2}\left(  \kappa\right)
)S_{2\ell}\left(  \kappa\right)  +s_{2\ell+2}\left(  \kappa\right)
-t_{2\ell+2}\left(  \kappa\right)  \left(  L_{2\ell}\left(  \kappa
^{-1}\right)  +T_{2\ell}\left(  \kappa^{-1}\right)  \right)  ,\\
T_{2\ell+2}\left(  \kappa\right)   &  =(1+s_{2\ell+2}\left(  \kappa\right)
)T_{2\ell}\left(  \kappa\right)  +s_{2\ell+2}\left(  \kappa\right)  L_{2\ell
}\left(  \kappa\right)  +t_{2\ell+2}\left(  \kappa\right)  S_{2\ell}\left(
\kappa^{-1}\right)  .
\end{align*}

\end{lemma}

We first consider the first term on the right hand side of \eqref{pprodlem}.
The following lemma will be useful.

\begin{lemma}
\label{lem:defRl2} For a constant $C>0$ only depending on $c_{\min},c_{\max}$
the following holds for $1<2\ell\leq L$
\[
L_{2\ell}(\mu^{-1})=\frac{z_{2\ell}}{c_{2\ell+1}}+R_{2\ell}^{3},\qquad
|R_{2\ell}^{3}|\leq Cz_{2\ell}^{3}.
\]

\end{lemma}

\begin{proof}
This follows directly form the definitions of $t_{\ell}$ and $L_{2\ell}$ and
using a Taylor expansion with respect to $\delta_{1},\dots, \delta_{2\ell}$.
\end{proof}

In view of $\beta_{0,0}:=1$, i.e., $\left(  \beta_{0,0}^{R},\beta_{0,0}%
^{I}\right)  =\left(  1,0\right)  $ we observe that the first term in the
product of $P_{2k}$ (cf. \eqref{pprodlem}) applied to the initial coefficient
$\beta_{0,0}$ yields%
\begin{equation}%
\begin{bmatrix}
1 & L_{2\ell}\left(  \mu\right) \\
-L_{2\ell}\left(  \mu^{-1}\right)  & 1
\end{bmatrix}%
\begin{pmatrix}
1\\
0
\end{pmatrix}
=%
\begin{pmatrix}
1\\
-\frac{z_{2\ell}}{c_{2\ell+1}}-R_{2\ell}^{3}%
\end{pmatrix}
. \label{formprincpart}%
\end{equation}

Next, we focus on the second term of \eqref{pprodlem}. The goal is to show
that the matrix entries are of higher order than the first term with respect
to $z_{\ell}$. We start with an estimate of $s_{\ell},t_{\ell}$ and $L_{2\ell
}$.

\begin{lemma}
\label{lem:estRZ} The following estimates hold%
\begin{align*}
\left\vert s_{\ell}\left(  \kappa\right)  \right\vert  &  \leq R_{\ell}%
^{2}:=C_{1}\left(  \delta_{\ell-1}^{2}+\delta_{\ell}^{2}\right)  ,\\
\left\vert t_{\ell}\left(  \kappa\right)  \right\vert  &  \leq R_{\ell}%
^{1}:=C_{1}\left(  \delta_{\ell-1}+\delta_{\ell}\right)  ,\\
\left\vert L_{2\ell}\left(  \kappa\right)  \right\vert  &  \leq Z_{2\ell
}:=C_{1}z_{2\ell},
\end{align*}
for $\kappa\in\{\mu,\mu^{-1}\}$ and the constant $C_{1}=\frac{\max\{1,c_{\max
}\}}{c_{\min}^{2}}$ as in \eqref{defC1}. In particular, condition
\eqref{defC1} implies%
\begin{equation}
0<R_{\ell}^{1}\leq1/2. \label{estRl1cond}%
\end{equation}

\end{lemma}

\begin{proof}
This is a simple consequence of the definitions of $s_{\ell}, t_{\ell}$ and
$L_{2\ell}$.
\end{proof}

The estimates in Lemma \ref{lem:estRZ} can be used to define the majorant of
$S_{2\ell}\left(  \kappa\right)$, $T_{2\ell}\left(  \kappa\right)  $ by%
\[%
\begin{array}
[c]{rl}%
\mathbf{v}_{\ell} & :=\left(  \check{S}_{2\ell}\left(  \kappa\right)
,\check{T}_{2\ell}\left(  \kappa\right)  \right)  ^{\intercal},\\
\check{S}_{2\ell}\left(  \kappa\right)  & :=\max\left\{  \left\vert S_{2\ell
}\left(  \kappa\right)  \right\vert ,\left\vert S_{2\ell}\left(  \kappa
^{-1}\right)  \right\vert \right\}  ,\\
\check{T}_{2\ell}\left(  \kappa\right)  & :=\max\left\{  \left\vert T_{2\ell
}\left(  \kappa\right)  \right\vert ,\left\vert T_{2\ell}\left(  \kappa
^{-1}\right)  \right\vert \right\}
\end{array}
\]
via the affine recursion
\[
\mathbf{v}_{\ell+1}=\mathbf{Q}_{\ell}\mathbf{v}_{\ell}+\mathbf{R}_{\ell}%
\]
for
\[
\mathbf{Q}_{\ell}:=%
\begin{bmatrix}
1+R_{2\ell+2}^{2} & R_{2\ell+2}^{1}\\
R_{2\ell+2}^{1} & 1+R_{2\ell+2}^{2}%
\end{bmatrix}
\quad\text{and\quad}\mathbf{R}_{\ell}:=%
\begin{bmatrix}
R_{2\ell+2}^{2}+R_{2\ell+2}^{1}Z_{2\ell}\\
R_{2\ell+2}^{2}Z_{2\ell}%
\end{bmatrix}
,
\]
where $R_{2\ell+2}^{1},R_{2\ell+2}^{2},Z_{2\ell}$ are defined in Lemma
\ref{lem:estRZ}. The recursion can be resolved and we get
\begin{equation}
\mathbf{v}_{\ell+1}=\mathbf{P}_{0,\ell}\mathbf{v}_{1}+\sum_{j=1}^{\ell
}\mathbf{P}_{j,\ell}\mathbf{R}_{j},\quad\text{with\quad}\mathbf{P}_{j,\ell
}:=\mathbf{Q}_{\ell}\cdot\ldots\cdot\mathbf{Q}_{j+1}.
\label{resolvedrecursion}%
\end{equation}
A simple algebraic calculus (cf. Lemma \ref{lem:A-prod}) yields%
\[
\mathbf{P}_{j,\ell}=%
\begin{bmatrix}
p_{j,\ell}^{(1)} & p_{j,\ell}^{(2)}\\
p_{j,\ell}^{(2)} & p_{j,\ell}^{(1)}%
\end{bmatrix}
,
\]
with
\begin{align*}
p_{j,\ell}^{(1)}  &  =\frac{1}{2}\left(  \prod\limits_{k=j+1}^{\ell}\left(
1+R_{2k+2}^{2}+R_{2k+2}^{1}\right)  +\prod\limits_{k=j+1}^{\ell}\left(
1+R_{2k+2}^{2}-R_{2k+2}^{1}\right)  \right)  ,\\
p_{j,\ell}^{(2)}  &  =\frac{1}{2}\left(  \prod\limits_{k=j+1}^{\ell}\left(
1+R_{2k+2}^{2}+R_{2k+2}^{1}\right)  -\prod\limits_{k=j+1}^{\ell}\left(
1+R_{2k+2}^{2}-R_{2k+2}^{1}\right)  \right)  .
\end{align*}
We may apply Lemma \ref{lem:A-estprod} with $a_{k}=1+R_{2k+2}^{2}$ and
$b_{k}=R_{2k+2}^{1}$ by taking into account \eqref{estRl1cond}. This leads to%
\[%
\begin{array}
[c]{c}%
\left\vert p_{j,\ell}^{(1)}\right\vert \leq\exp\left(  \sum_{k=j+1}^{\ell
}R_{2k+2}^{2}\right)  \cosh\left(  \sum_{\ell=j+1}^{k}R_{2k+2}^{1}\right)  ,\\
\left\vert p_{j,\ell}^{(2)}\right\vert \leq\exp\left(  \sum_{k=j+1}^{\ell
}R_{2k+2}^{2}\right)  \sinh\left(  \sum_{k=j+1}^{\ell}R_{2k+2}^{1}\right)  .
\end{array}
\]

\begin{lemma}
\label{LemmaPest1} It holds
\begin{equation}%
\begin{split}
\left\vert p_{j,\ell}^{(1)}\right\vert  &  \leq1+C_2\left(  z_{2\ell+2}%
-z_{2j+2}\right)  ^{2},\\
\left\vert p_{j,\ell}^{(2)}\right\vert  &  \leq C_2\left(  z_{2\ell+2}%
-z_{2j+2}\right)  ,
\end{split}
\label{Pest1}%
\end{equation}
for some constant $C_2>0$ depending only on $c_{\min},~c_{\max}$.
\end{lemma}

\begin{proof}
By Lemma \ref{lem:estRZ}, we have
\[%
\begin{split}
\sum_{k=j+1}^{\ell}R_{2k+2}^{1}  &  \leq C_{1}\sum_{k=j+1}^{\ell}\left(
\delta_{2k+1}+\delta_{2k+2}\right)  \leq\frac{C_{1}}{c_{\min}}\left(
z_{2\ell+2}-z_{2j+2}\right)  ,\\
\sum_{k=j+1}^{\ell}R_{2k+2}^{2}  &  \leq C_{1}\left(  \sum_{k=j+1}^{\ell
}\left(  \delta_{2k+1}+\delta_{2k+2}\right)  \right)  ^{2}\leq\frac{C_{1}%
}{c_{\min}^{2}}\left(  z_{2\ell+2}-z_{2j+2}\right)  ^{2},
\end{split}
\]
which leads to the assertion.
\end{proof}

\begin{lemma}
\label{LemPjkRj}The sum in \eqref{resolvedrecursion} can be estimated
componentwise by%
\begin{align*}
\sum_{j=1}^{\ell}\left\vert \left(  \mathbf{P}_{j,\ell}\mathbf{R}_{j}\right)
_{1}\right\vert  &  \leq C_{3}z_{2\ell+2}^{2}\\
\sum_{j=1}^{\ell}\left\vert \left(  \mathbf{P}_{j,\ell}\mathbf{R}_{j}\right)
_{2}\right\vert  &  \leq C_{3}z_{2\ell+2}^{3}%
\end{align*}
for a constant $C_{3}>0$ depending only on $c_{\min},~c_{\max}$.
\end{lemma}

\begin{proof}
For the components of $\mathbf{R}_{j}$ we get from Lemma \ref{lem:estRZ}
\begin{align*}
\left\vert R_{2j+2}^{2}+R_{2j+2}^{1}Z_{2j}\right\vert  &  \leq C_{1}\left(
\delta_{2j+1}^{2}+\delta_{2j+2}^{2}\right)  +C_{1}\left(  \delta_{2j+1}%
+\delta_{2j+2}\right)  Z_{2j},\\
\left\vert R_{2j+2}^{2}Z_{2j}\right\vert  &  \leq C_{1}\left(  \delta
_{2j+1}^{2}+\delta_{2j+2}^{2}\right)  Z_{2j}.
\end{align*}
We use \eqref{Pest1} to get%
\begin{align*}
\left\vert \left(  \mathbf{P}_{j,\ell}\mathbf{R}_{j}\right)  _{1}\right\vert
&  \leq C\left(  1+Z_{2\ell+2}^{2}\right)  \left(  \left(  \delta
_{2j+1}+\delta_{2j+2}\right)  Z_{2\ell+2}+\left(  \delta_{2j+1}^{2}%
+\delta_{2j+2}^{2}\right)  \right)  ,\\
\left\vert \left(  \mathbf{P}_{j,\ell}\mathbf{R}_{j}\right)  _{2}\right\vert
&  \leq CZ_{2\ell+2}\left(  Z_{2\ell+2}\left(  \delta_{2j+1}+\delta
_{2j+2}\right)  +\left(  1+Z_{2\ell+2}^{2}\right)  \left(  \delta_{2j+1}%
^{2}+\delta_{2j+2}^{2}\right)  \right)
\end{align*}
for a constant $C >1$ depending only on $c_{\min},c_{\max}$. A summation
over $j$ leads to the assertion.
\end{proof}

Finally, we can combine Lemma \ref{LemPjkRj} and the fact that
\[
\check{S}_{2}\left(  \mu\right)  =\max\left\{  \left\vert S_{2}\left(
\mu\right)  \right\vert ,\left\vert S_{2}\left(  \mu^{-1}\right)  \right\vert
\right\}  \leq C_{1}(\delta_{1}^{2}+\delta_{2}^{2})\leq\check{C}_{1}z_{2\ell
}^{2}%
\]
for all $1\leq\ell\leq L$ to estimate $\check{S}_{2\ell}$ and $\check
{T}_{2\ell}$ given by the recursion \eqref{resolvedrecursion} for $\kappa
\in\{\mu,\mu^{-1}\}$ by%
\begin{equation}%
\begin{array}
[c]{c}%
\left\vert S_{2\ell}(\kappa)\right\vert \leq\check{S}_{2\ell}(\mu)\leq
Cz_{2\ell}^{2},\\
\left\vert T_{2\ell}(\kappa)\right\vert \leq\check{T}_{2\ell}(\mu)\leq
Cz_{2\ell}^{3}.
\end{array}
\label{eq:STest}%
\end{equation}
We summarize our findings in the next lemma, including the estimate of
$\beta_{0,\ell}$ for odd indices $\ell$.

\begin{lemma}
\label{lem:RepBeta} It holds that
\[
\operatorname{Re}\beta_{0,\ell}=1+\mathfrak{s}_{\ell},\qquad\operatorname{Im}%
\beta_{0,\ell}=-\frac{z_{\ell}}{c_{\ell+1}}-\mathfrak{t}_{\ell},
\]
for some $\mathfrak{s}_{\ell}$, $\mathfrak{t}_{\ell}\in\mathbb{R}$, bounded
by
\begin{equation}
\left\vert \mathfrak{s}_{\ell}\right\vert \leq C\left(  \frac{\omega x_{\ell}%
}{c_{\ell}}\right)  ^{2},\qquad\left\vert \mathfrak{t}_{\ell}\right\vert \leq
C\left(  \frac{\omega x_{\ell}}{c_{\ell}}\right)  ^{3}, \label{sltlest}%
\end{equation}
for a constant $C>0$ only depending on $c_{\min},c_{\max}$.
\end{lemma}

\begin{proof}
For the even entries the claim follows from the straightforward combination of
\eqref{eq:betaRealImag}, \eqref{pprodlem}, \eqref{formprincpart},
\eqref{eq:STest}. For the odd entries, we first use \eqref{eq:betaRealImagOdd}
and compute
\begin{align*}
\operatorname{Re}\left(  \beta_{0,2\ell+1}\right)   &  =\cos(\delta_{2\ell
+1})\operatorname{Re}\left(  \beta_{0,2\ell}\right)  +\sin(\delta_{2\ell
+1})\operatorname{Im}\left(  \beta_{0,2\ell}\right) \\
&  =1+\mathfrak{s}_{2\ell+1},
\end{align*}
for some $\mathfrak{s}_{2\ell+1}$ with $\left\vert \mathfrak{s}_{2\ell
+1}\right\vert \leq Cz_{2\ell+1}^{2}$ by a Taylor argument. On the other hand
we also know%
\[
\operatorname{Re}(\beta_{0,2\ell})=1+\mathfrak{s}_{2\ell},\qquad
\operatorname{Im}(\beta_{0,2\ell})=-\frac{z_{2\ell}}{c_{2\ell+1}}%
-\mathfrak{t}_{2\ell},
\]
for some $\mathfrak{s}_{2\ell}$, $\mathfrak{t}_{2\ell}\in\mathbb{R}$, bounded
by
\[
\left\vert \mathfrak{s}_{2\ell}\right\vert \leq C\left(  \frac{\omega
x_{2\ell}}{c_{2\ell}}\right)  ^{2},\qquad\left\vert \mathfrak{t}_{2\ell
}\right\vert \leq C\left(  \frac{\omega x_{2\ell}}{c_{2\ell}}\right)  ^{3}.
\]
For the imaginary part, we employ again \eqref{eq:betaRealImagOdd} and
\eqref{sltlest} for (even) $2\ell$ to get%
\begin{align*}
\operatorname{Im}&\beta_{0,2\ell+1}    =-\mu\sin(\delta_{2\ell+1}%
)\beta_{0,2\ell}^{R}+\mu\cos(\delta_{2\ell+1})\beta_{0,2\ell}^{I}\\
&  =-\mu\sin(\delta_{2\ell+1})\left(  1+\mathfrak{s}_{2\ell}\right)  +\mu
\cos(\delta_{2\ell+1})\left(  -\frac{z_{2\ell}}{c_{2\ell+1}}-\mathfrak{t}%
_{2\ell}\right) \\
&  =-\frac{z_{2\ell+1}}{c_{2\ell}}+\mu\frac{\tilde{\delta}_{2\ell+1}^{3}}%
{6}-\mu\delta_{2\ell+1}\mathfrak{s}_{2\ell}+\mu\frac{\left(  \delta_{2\ell
+1}^{\prime}\right)  ^{2}}{2}\frac{z_{2\ell}}{c_{2\ell+1}}+\mu\cos
(\delta_{2\ell+1})\left(  -\mathfrak{t}_{2\ell}\right)  ,
\end{align*}
for some $0\leq\tilde{\delta}_{2\ell+1},\delta_{2\ell+1}^{\prime}\leq
\delta_{2\ell+1}$. By using \eqref{sltlest} (for (even) $2\ell$) the assertion follows.
\end{proof}

\begin{proposition}
\label{prop:refinedestbeta} For all $\ell\leq L$ it holds
\[
\operatorname{Im}\left(  e^{\operatorname*{i}\frac{z_{\ell}}{c_{\ell+1}}}%
\beta_{0,\ell}\right)  \leq Cz_{\ell}^{3}.
\]
The constant $C$ depends only on $c_{\min},~c_{\max}$ but is independent of
$L$.
\end{proposition}

\begin{proof}
Using Lemma \ref{lem:RepBeta}, we compute%
\begin{align*}
\operatorname{Im}  &  \left(  e^{\operatorname*{i}\frac{z_{\ell}}{c_{\ell+1}}%
}\beta_{0,\ell}\right)  =\cos\left(  \frac{z_{\ell}}{c_{\ell+1}}\right)
\operatorname{Im}\left(  \beta_{0,\ell}\right)  +\sin\left(  \frac{z_{\ell}%
}{c_{\ell+1}}\right)  \operatorname{Re}\left(  \beta_{0,\ell}\right) \\
&  =-\cos\left(  \frac{z_{\ell}}{c_{\ell+1}}\right)  \left(  \frac{z_{\ell}%
}{c_{\ell+1}}+\mathfrak{t}_{\ell}\right)  +\sin\left(  \frac{z_{\ell}}%
{c_{\ell+1}}\right)  \left(  1+\mathfrak{s}_{\ell}\right) \\
&  =-\frac{1}{6}\left(  \frac{z_{\ell}^{\prime}}{c_{\ell+1}}\right)
^{3}-\mathfrak{t}_{\ell}-\frac{1}{2}\left(  \frac{\tilde{z}_{\ell}}{c_{\ell
+1}}\right)  ^{2}\left(  \frac{z_{\ell}}{c_{\ell+1}}+\mathfrak{t}_{\ell
}\right)  +\mathfrak{s}_{\ell}\sin\left(  \frac{z_{\ell}}{c_{\ell+1}}\right)
\end{align*}
for some $0\leq\tilde{z}_{\ell},z_{\ell}^{\prime}\leq z_{\ell}$. Using the
estimates \eqref{sltlest} yields the assertion.
\end{proof}

\subsection{Proof of Theorem \ref{thm:finalStab}\label{sec:proofstab}}

From Proposition \ref{th:uplowbeta}, Proposition \ref{prop:refinedestbeta},
the representation \eqref{Greenrep}, definition
\eqref{fullinteriorsysfinal} and the results of Theorem \ref{th:final_rep}, we conclude
\begin{align*}
|A_{\ell+1}|  &  =\left\vert \hat{g}_{0}\right\vert \frac{\omega}{c_{N}%
}\left\vert h_{0}^{(1)}\left(  \frac{\omega}{c_{N}}\right)  \right\vert
\left\vert \frac{\operatorname{Im}\left(  \operatorname{e}^{\operatorname*{i}%
\frac{z_{\ell}}{c_{\ell+1}}}\beta_{0,\ell}\right)  }{\operatorname{e}%
^{\operatorname*{i}\frac{z_{n}}{c_{n+1}}}\beta_{0,n}}\right\vert
\lesssim\left\vert \hat{g}_{0}\right\vert \alpha^{\omega}\min\{z_{\ell}%
^{3},1\},\\
|B_{\ell}|  & =\left\vert \hat{g}_{0}\right\vert \frac{\omega}{c_{N}%
}\left\vert h_{0}^{(1)}\left(  \frac{\omega}{c_{N}}\right)  \right\vert
\left\vert \frac{\operatorname{e}^{\operatorname*{i}\frac{z_{\ell-1}}{c_{\ell}}}\beta_{0,\ell-1}}{\operatorname{e}^{\operatorname*{i}\frac{z_{n}}{c_{n+1}}%
}\beta_{0,n}}\right\vert \lesssim\left\vert \hat{g}_{0}\right\vert
\alpha^{\omega},
\end{align*}
for all $1\leq\ell\leq n$ and some $\alpha>1$. We note that we also used
\eqref{modh0j0}, to estimate the term $\frac{\omega}{c_{N}}\left\vert
h_{0}^{(1)}\left(  \frac{\omega}{c_{N}}\right)  \right\vert $. Moreover we
recall \eqref{Amn1}, \eqref{BmnN} that
\[
|A_{1}|=0=z_{0},\qquad|B_{N}|=\left\vert \frac{\omega}{c_{N}}%
h_{0}^{(1)}\left(  \frac{\omega}{c_{N}}\right)  \right\vert \left\vert \hat
{g}_{0}\right\vert =\left\vert \hat{g}_{0}\right\vert .
\]
Now we apply these estimates of $A_{\ell},B_{\ell}$ to \eqref{eq:l2estsol} and
\eqref{eq:gradl2estsol}. The combination of this with estimates of Hankel and
Bessel functions in Lemma \ref{lem:integralhankelbessel} yields%
\begin{align*}
\left\Vert u\right\Vert _{\mathcal{H}}^{2}  &  \leq2\max_{2\leq j\leq N}%
A_{j}^{2}\left(  2+\left(  \frac{c_{j}}{z_{j-1}}\right)  ^{2}\right)
+2\max_{1\leq j\leq N}B_{j}^{2}\left(  \frac{16z_{j}^{4}}{\left(  2c_{j}%
^{2}+z_{j}^{2}\right)  ^{2}}+\left(  \frac{2z_{j}}{c_{j}+z_{j}}\right)
^{2}\right) \\
&  \lesssim\left\vert \hat{g}_{0}\right\vert ^{2}\alpha^{\omega}\left\{
\max_{2\leq j\leq N}\left(  \min\{z_{j}^{6},1\}\left(  2+\left(  \frac{c_{j}%
}{z_{j-1}}\right)  ^{2}\right)  \right) \right.\\
& \phantom{\lesssim\left\vert \hat{g}_{0}\right\vert ^{2}\alpha^{\omega}}+\left.\max_{1\leq j\leq N}\left(
\frac{16z_{j}^{4}}{\left(  2c_{j}^{2}+z_{j}^{2}\right)  ^{2}}+\left(
\frac{2z_{j}}{c_{j}+z_{j}}\right)  ^{2}\right)  \right\} \\
&  \lesssim\left\vert \hat{g}_{0}\right\vert ^{2}\alpha^{\omega}.
\end{align*}

Next we will prove the pointwise estimates in Theorem \ref{thm:finalStab}. For
$r\in\tau_{j}$, $2\leq j\leq N$, we obtain%
\begin{align*}
\left\vert \hat{u}\left(  r,\boldsymbol{\xi}\right)  \right\vert  &  \lesssim\left\vert \hat{g}_{0}\right\vert
\alpha^{\omega}\left(  \min\{z_{j-1}^{3},1\}\dfrac{c_{j}}{z_{j-1}}+\dfrac
{2}{1+\dfrac{z_{j-1}}{c_{j}}}\right) \\
&  \lesssim\dfrac{\alpha^{\omega}}{1+z_{j-1}}\left\vert \hat{g}_{0}\right\vert
\leq\alpha^{\omega}\left\vert \hat{g}_{0}\right\vert
\end{align*}
and in a similar fashion
\[
\left\vert \left(  \nabla u\right)  \circ\psi\left(  r,\mathbf{\xi} \right)
\right\vert \lesssim\left\vert \hat{g}_{0}\right\vert \alpha^{\omega} \left(
\min\{z_{j-1}^{3},1\} \left(  \frac{c_{j}}{\omega r}\left(  1+ \frac{c_{j}%
}{\omega r}\right)  \right)  + \frac{\omega r}{c_{j}}\right)  \lesssim
r\tilde{\alpha}^{\omega}\left\vert \hat{g}_{0}\right\vert
\]
For the first interval $\tau_{1}$ we use that $A_{1}=0$ and obtain the same estimates.

\section{Proof of the Representation of the Green's Operator (Theorem \ref{th:final_rep})}

\label{sec:proofRep} We denote by $\mathbf{M}_{m}^{\left(  2n,i,j\right)  }$
the matrix which arises by removing the $i$-th row and the $j$-th column of
$\mathbf{M}_{m}^{\left(  2n\right)  }$. According to Cramer's rule, we have%
\[
\left(  \mathbf{\hat{M}}_{m}^{\left(  2n\right)  }\right)  _{i,j}^{-1}=\left(
-1\right)  ^{i+j}\frac{\det\mathbf{\hat{M}}_{m}^{\left(  2n,j,i\right)  }%
}{\det\mathbf{\hat{M}}_{m}^{\left(  2n\right)  }}.
\]
From \eqref{defMhat2n} and the well-known recursion formula for determinants
of tri-diagonal matrices we get
\begin{equation}
\det\mathbf{\hat{M}}_{m}^{\left(  2n,2n,i\right)  }=\left(
{\displaystyle\prod\limits_{\ell=\left\lfloor \frac{i+1}{2}\right\rfloor
}^{n-1}}\left(  \mathbf{\hat{T}}_{m}^{\left(  \ell\right)  }\right)
_{2,1}\right)  \left(  {\displaystyle\prod\limits_{\ell=\left\lfloor
\frac{i+2}{2}\right\rfloor }^{n}}\left(  \mathbf{\hat{S}}_{m}^{\left(
\ell\right)  }\right)  _{1,2}\right)  \det\mathbf{\hat{M}}_{m}^{\left(
i-1\right)  }. \label{eq:detM}%
\end{equation}
Note that
\[
\det\mathbf{\hat{S}}_{m}^{\left(  \ell\right)  }=\frac{w_{m,\ell,\ell+1,\ell
}^{2,1}}{w_{m,\ell+1,\ell,\ell}^{2,1}}%
\]
which is well defined (cf. proof of \cite[Lemma 1]{HansenPoignardVogelius2007}%
). Next, we express the determinant $\det\mathbf{\hat{M}}_{m}^{\left(
2n\right)  }$ in a recursive way. For $n\in\mathbb{N}_{\geq1}$ and
$q\in\left\{  1,2\right\}  $, let (with Kronecker's delta $\delta_{i,j}$)
\[
W_{m,\ell,q}:=\left\{
\begin{array}
[c]{ll}%
\delta_{1,q} & \text{for }\ell=0,\\
\det\left[
\begin{array}
[c]{ll}%
W_{m,\ell-1,1} & w_{m,\ell,\ell+1,\ell}^{1,q}\\
W_{m,\ell-1,2} & w_{m,\ell,\ell+1,\ell}^{2,q}%
\end{array}
\right]  & \text{for }\ell\geq1.
\end{array}
\right.
\]

\begin{lemma}
\label{lem:detM1} Let $n\geq1$. We have
\[
\det\mathbf{\hat{M}}_{m}^{\left(  2n\right)  }=\frac{W_{m,n,1}}{\prod
\limits_{\ell=1}^{n}w_{m,\ell+1,\ell,\ell}^{2,1}}%
\]
and
\[
\det\mathbf{\hat{M}}_{m}^{\left(  2n-1\right)  }=\frac{-W_{m,n,2}%
}{\displaystyle\prod\limits_{\ell=1}^{n}w_{m,\ell+1,\ell,\ell}^{2,1}}.
\]

\end{lemma}

\begin{proof}
By induction: The case $n=1$ can be easily checked for both cases. Then, we
have
\begin{align*}
\det\mathbf{\hat{M}}_{m}^{\left(  2(n+1)\right)  }=  &  \det\mathbf{\hat{S}%
}_{m}^{\left(  n+1\right)  }\det\mathbf{\hat{M}}_{m}^{\left(  2n\right)
}+\left(  \mathbf{\hat{S}}_{m}^{\left(  n+1\right)  }\right)  _{2,2}%
\det\mathbf{\hat{M}}_{m}^{\left(  2n-1\right)  }\\
=  &  \frac{w_{m,n+1,n+2,n+1}^{2,1}}{w_{m,n+2,n+1,n+1}^{2,1}}\frac{W_{m,n,1}%
}{\prod\limits_{\ell=1}^{n}w_{m,\ell+1,\ell,\ell}^{2,1}}-\frac
{w_{m,n+1,n+2,n+1}^{1,1}}{w_{m,n+2,n+1,n+1}^{2,1}}\frac{W_{m,n,2}}%
{\prod\limits_{\ell=1}^{n}w_{m,\ell+1,\ell,\ell}^{2,1}}\\
=  &  \frac{W_{m,n+1,1}}{\prod\limits_{\ell=1}^{n+1}w_{m,\ell+1,\ell,\ell
}^{2,1}}%
\end{align*}
and, in turn,%
\begin{align*}
\det\mathbf{\hat{M}}_{m}^{\left(  2n+1\right)  }=  &  \left(  \mathbf{\hat{S}%
}_{m}^{\left(  n+1\right)  }\right)  _{1,1}\det\mathbf{\hat{M}}_{m}^{\left(
2n\right)  }+\det\mathbf{\hat{M}}_{m}^{\left(  2n-1\right)  }\\
=  &  \frac{w_{m,n+2,n+1,n+1}^{2,2}}{w_{m,n+2,n+1,n+1}^{2,1}}\frac{W_{m,n,1}%
}{\prod\limits_{\ell=1}^{n}w_{m,\ell+1,\ell,\ell}^{2,1}}-\frac{W_{m,n,2}%
}{\prod\limits_{\ell=1}^{n}w_{m,\ell+1,\ell,\ell}^{2,1}}\\
=  &  \frac{-W_{m,n+1,2}}{\prod\limits_{\ell=1}^{n+1}w_{m,\ell+1,\ell,\ell
}^{2,1}}.
\end{align*}

\end{proof}

In the next step we will derive a representation of $W_{m,n,1}$. We define the
sequence $\left(  \tilde{\beta}_{m,\ell}\right)  _{\ell=0}^{n}$ by%
\begin{equation}%
\begin{split}
\tilde{\beta}_{m,0}  &  :=1,\\
\tilde{\beta}_{m,\ell}  &  :=\frac{\tilde{\gamma}_{m,\ell}^{+}}{2}\left(
\tilde{\beta}_{m,\ell-1}-\left(  -1\right)  ^{\ell}\tilde{q}_{m,\ell}%
\overline{\tilde{\beta}_{m,\ell-1}}\right)
\end{split}
\label{eq:recursion1}%
\end{equation}
using the definition in \eqref{def:gammaqnew} for $\tilde{\gamma}_{m,\ell}%
^{+}$ and $\tilde{q}_{m,\ell}$.

\begin{lemma}
\label{lem:Wmn}For $n\in\mathbb{N}$ it holds
\[
W_{m,n,1}=\left(  -1\right)  ^{n}\tilde{\beta}_{m,n}\qquad\text{and}\qquad
W_{m,n,2}:=-\frac{\overline{\tilde{\beta}_{m,n}}-(-1)^{n}\tilde{\beta}_{m,n}%
}{2}.
\]

\end{lemma}

\begin{proof}
By using the definition of $w_{m,n,n+1,n}^{2,1}$ and the relation $f_{m,2}=\frac{1}{2}\left(f_{m,1}+\overline{f_{m,1}}\right)$ it is easy to verify%
\[%
\begin{array}
[c]{ll}%
w_{m,\ell,\ell+1,\ell}^{1,1}=-\tilde{\gamma}_{m,\ell}^{-}, & w_{m,\ell
,\ell+1,\ell}^{1,2}=-\frac{1}{2}\left(  \overline{\tilde{\gamma}_{m,\ell}^{+}%
}+\tilde{\gamma}_{m,\ell}^{-}\right)  ,\\
w_{m,\ell,\ell+1,\ell}^{2,1}=-\frac{1}{2}\left(  \tilde{\gamma}_{m,\ell}%
^{+}+\tilde{\gamma}_{m,\ell}^{-}\right)  , & w_{m,\ell,\ell+1,\ell}%
^{2,2}=-\frac{1}{4}\left(  \tilde{\gamma}_{m,\ell}^{+}+\tilde{\gamma}_{m,\ell
}^{-}+\overline{\tilde{\gamma}_{m,\ell}^{+}}+\overline{\tilde{\gamma}_{m,\ell
}^{-}}\right)  .
\end{array}
\]
Now we start to prove the statement by induction. For $n=1$ we have
\begin{align*}
W_{m,1,1}=  &  w_{m,1,2,1}^{2,1}=-\frac{1}{2}\left(  \tilde{\gamma}_{m,1}%
^{+}+\tilde{\gamma}_{m,1}^{-}\right) \\
=  &  -\frac{\tilde{\gamma}_{m,1}^{+}}{2}\left(  1+\tilde{q}_{m,1}\right)
=-\tilde{\beta}_{m,1},\\
W_{m,1,2}=  &  w_{m,1,2,1}^{2,2}=-\frac{1}{4}\left(  \tilde{\gamma}_{m,1}%
^{+}+\tilde{\gamma}_{m,1}^{-}+\overline{\tilde{\gamma}_{m,1}^{+}}%
+\overline{\tilde{\gamma}_{m,1}^{-}}\right) \\
=  &  -\frac{1}{2}\left(  \frac{\tilde{\gamma}_{m,1}^{+}}{2}\left(
1+\tilde{q}_{m,1}\right)  +\frac{\overline{\tilde{\gamma}_{m,1}^{+}}}%
{2}\left(  1+\overline{\tilde{q}_{m,1}}\right)  \right) \\
=  &  -\frac{\overline{\tilde{\beta}_{m,1}}+\tilde{\beta}_{m,1}}{2}.
\end{align*}
Assume the statement is true for $n-1$. We have from the definition of
$W_{m,n,1}$
\begin{align*}
W_{m,n,1}=  &  \left(  W_{m,n-1,1}w_{m,n,n+1,n}^{2,1}-W_{m,n-1,2}%
w_{m,n,n+1,n}^{1,1}\right) \\
=  &  \frac{1}{2}\left(  (-1)^{n}\tilde{\beta}_{m,n-1}\left(  \tilde{\gamma
}_{m,n}^{+}+\tilde{\gamma}_{m,n}^{-}\right)  -\overline{\tilde{\beta}_{m,n-1}%
}\tilde{\gamma}_{m,n}^{-}-(-1)^{n}\tilde{\beta}_{m,n-1}\tilde{\gamma}%
_{m,n}^{-}\right) \\
=  &  (-1)^{n}\tilde{\beta}_{m,n}.
\end{align*}
For $W_{m,n,2}$, we compute%
\begin{align*}
W_{m,{n},2}=  &  \left(  W_{m,{n}-1,1}w_{m,n,n+1,n}^{2,2}-W_{m,n-1,2}%
w_{m,n,n+1,n}^{1,2}\right) \\
=  &  \frac{1}{4}\left(  (-1)^{n}\tilde{\beta}_{m,n-1}\left(  \tilde{\gamma
}_{m,n}^{+}+\overline{\tilde{\gamma}_{m,n}^{-}}\right)  -\overline
{\tilde{\beta}_{m,n-1}}\left(  \overline{\tilde{\gamma}_{m,n}^{+}}%
+\tilde{\gamma}_{m,n}^{-}\right)  \right) \\
=  &  -\frac{1}{2}\left(  \overline{\tilde{\beta}_{m,n}}-(-1)^{n}\tilde{\beta
}_{m,n}\right)  .
\end{align*}%
\end{proof}

The combination of representation in \eqref{eq:detM} and applying Lemma
\ref{lem:detM1} and \ref{lem:Wmn} yields for odd $i=2\ell-1$, $1\leq\ell\leq
n$
\begin{equation}%
\begin{array}
[c]{rl}%
\left(  \mathbf{\hat{M}}_{m}^{\left(  2n\right)  }\right)  _{2\ell-1,2n}^{-1}
& =-\dfrac{\det\mathbf{\hat{M}}_{m}^{\left(  2n,2n,2\ell-1\right)  }}%
{\det\mathbf{\hat{M}}_{m}^{\left(  2n\right)  }}=\left(  -1\right)  ^{n-\ell
}\dfrac{\left(  {\displaystyle\prod\limits_{k=\ell}^{n}}\mathbf{\hat{S}}%
_{1,2}^{\left(  k\right)  }\right)  \det\mathbf{\hat{M}}_{m}^{\left(
2(\ell-1)\right)  }}{\det\mathbf{\hat{M}}_{m}^{\left(  2n\right)  }}\\
& =-\left(  {\displaystyle\prod\limits_{k=\ell}^{n}}w_{m,k+1,k+1,k}%
^{1,2}\right)  \dfrac{\tilde{\beta}_{m,\ell-1}}{\tilde{\beta}_{m,n}}.
\end{array}
\label{eq:Mbetaa}%
\end{equation}
For $i=2\ell$ even, $1\leq\ell\leq n$, we compute%
\begin{align}
\left(  \mathbf{\hat{M}}_{m}^{\left(  2n\right)  }\right)  _{2\ell,2n}^{-1}
&  =\frac{\det\mathbf{\hat{M}}_{m}^{\left(  2n,2n,2\ell\right)  }}%
{\det\mathbf{\hat{M}}_{m}^{\left(  2n\right)  }}=-\left(  -1\right)  ^{n-\ell
}{\displaystyle}\left(  {\prod\limits_{k=\ell+1}^{n}}w_{m,k+1,k+1,k}%
^{1,2}\right)  {\frac{W_{m,\ell,2}}{W_{m,n,1}}}\nonumber\\
&  =\frac{(-1)^{\ell}\overline{\tilde{\beta}_{m,\ell}}-\tilde{\beta}_{m,\ell}%
}{2\tilde{\beta}_{m,n}}{\displaystyle\prod\limits_{k=\ell+1}^{n}%
}w_{m,k+1,k+1,k}^{1,2}. \label{eq:Mbetab}%
\end{align}
Finally, we introduce
\[
\beta_{m,\ell}=\frac{\mathrm{e}^{-\operatorname*{i}\frac{z_{\ell}}{c_{\ell+1}%
}}}{\displaystyle\prod\limits_{k=1}^{\ell}w_{m,k+1,k+1,k}^{1,2}}\tilde{\beta
}_{m,\ell},\qquad0\leq\ell\leq n.
\]

Recalling Remark \ref{rmk:wronskian}, one can easily check that the recursion
for $\beta_{m,\ell}$ given in \eqref{eq:recursion2m0} follows from the
recursion of $\tilde{\beta}_{m,\ell}$ defined in \eqref{eq:recursion1}. The
representations \eqref{eq:repMGreen} are a direct consequence of the
definition of $\beta_{m,\ell}$ and equations \eqref{eq:Mbetaa} and
\eqref{eq:Mbetab}.

\section{Some Basic Facts from Linear Algebra}

\begin{lemma}
\label{lem:A-prod}Let%
\[
\mathbf{A}_{\ell}:=\left[
\begin{array}
[c]{cc}%
a_{\ell} & b_{\ell}\\
b_{\ell} & a_{\ell}%
\end{array}
\right]  .
\]
Then%
\[
\mathbf{A}_{k}\cdots\mathbf{A}_{1}=\frac{1}{2}\left[
\begin{array}
[c]{cc}%
\prod\limits_{\ell=1}^{k}\left(  a_{\ell}+b_{\ell}\right)  +\prod
\limits_{\ell=1}^{k}\left(  a_{\ell}-b_{\ell}\right)  & \prod\limits_{\ell
=1}^{k}\left(  a_{\ell}+b_{\ell}\right)  -\prod\limits_{\ell=1}^{k}\left(
a_{\ell}-b_{\ell}\right) \\
\prod\limits_{\ell=1}^{k}\left(  a_{\ell}+b_{\ell}\right)  -\prod
\limits_{\ell=1}^{k}\left(  a_{\ell}-b_{\ell}\right)  & \prod\limits_{\ell
=1}^{k}\left(  a_{\ell}+b_{\ell}\right)  +\prod\limits_{\ell=1}^{k}\left(
a_{\ell}-b_{\ell}\right)
\end{array}
\right]  .
\]

\end{lemma}

The proof of this lemma follows in a straightforward way by induction and is
skipped. To estimate the matrix product in Lemma \ref{lem:A-prod} for positive
coefficients we need the following lemma.

\begin{lemma}
\label{lem:A-estprod} For $0<b_{\ell}<a_{\ell}$, one has%
\begin{align*}
\prod\limits_{\ell=1}^{k}\left(  a_{\ell}+b_{\ell}\right)  +\prod
\limits_{\ell=1}^{k}\left(  a_{\ell}-b_{\ell}\right)   &  \leq2\left(
\prod\limits_{\ell=1}^{k}a_{\ell}\right)  \cosh\left(  \sum_{\ell=1}^{k}%
\frac{b_{\ell}}{a_{\ell}}\right)  ,\\
\prod\limits_{\ell=1}^{k}\left(  a_{\ell}+b_{\ell}\right)  -\prod
\limits_{\ell=1}^{k}\left(  a_{\ell}-b_{\ell}\right)   &  \leq2\left(
\prod\limits_{\ell=1}^{k}a_{\ell}\right)  \sinh\left(  \sum_{\ell=1}^{k}%
\frac{b_{\ell}}{a_{\ell}}\right)  .
\end{align*}

\end{lemma}

\begin{proof}
We compute
\[
\prod\limits_{\ell=1}^{k}\left(  a_{\ell}+b_{\ell}\right)  \pm\prod
\limits_{\ell=1}^{k}\left(  a_{\ell}-b_{\ell}\right)  =\left(  \prod
\limits_{\ell=1}^{k}a_{\ell}\right)  \left(  \prod\limits_{\ell=1}^{k}\left(
1+\frac{b_{\ell}}{a_{\ell}}\right)  \pm\prod\limits_{\ell=1}^{k}\left(
1-\frac{b_{\ell}}{a_{\ell}}\right)  \right)
\]
and recall that for $0<\gamma_{\ell}\in\mathbb{R}$
\begin{align*}
\prod\limits_{\ell=1}^{k}\left(  1+\gamma_{\ell}\right)  +\prod\limits_{\ell
=1}^{k}\left(  1-\gamma_{\ell}\right)   &  =2\sum_{\ell=0}^{\lfloor\frac{k}%
{2}\rfloor}\sum_{\substack{\alpha\in\{0,1\}^{k},\\|\alpha|=2\ell}%
}\prod\limits_{i=1}^{k}\gamma_{i}^{\alpha_{i}},\\
\prod\limits_{\ell=1}^{k}\left(  1+\gamma_{\ell}\right)  -\prod\limits_{\ell
=1}^{k}\left(  1-\gamma_{\ell}\right)   &  =2\sum_{\ell=0}^{\lfloor\frac
{k-1}{2}\rfloor}\sum_{\substack{\alpha\in\{0,1\}^{k},\\|\alpha|=2\ell+1}%
}\prod\limits_{i=1}^{k}\gamma_{i}^{\alpha_{i}}%
\end{align*}
Finally, we note that for any $n\in\mathbb{N}$
\[
\left(  \sum_{i=1}^{k}\gamma_{i}\right)  ^{n}=\sum_{\substack{\alpha
\in\mathbb{N}^{k}\\|\alpha|=n}}\frac{n!}{\alpha_{1}!\cdots\alpha_{k}!}%
\prod_{i=1}^{k}\gamma_{i}^{\alpha_{i}}\geq n!\sum_{\substack{\alpha
\in\{0,1\}^{k}\\|\alpha|=n}}\prod_{i=1}^{k}\gamma_{i}^{\alpha_{i}}.
\]
and
\begin{align*}
\cosh\left(  \sum_{i=1}^{k}\gamma_{i}\right)   &  =\sum_{\ell=0}^{\infty}%
\frac{\left(  \sum_{i=1}^{k}\gamma_{i}\right)  ^{2\ell}}{(2\ell)!}\geq
\sum_{n=0}^{n}\frac{\left(  \sum_{i=1}^{k}\gamma_{i}\right)  ^{2n}}{(2n)!},\\
\sinh\left(  \sum_{i=1}^{k}\gamma_{i}\right)   &  =\sum_{\ell=0}^{\infty}%
\frac{\left(  \sum_{i=1}^{k}\gamma_{i}\right)  ^{2\ell+1}}{(2\ell+1)!}\geq
\sum_{\ell=0}^{n}\frac{\left(  \sum_{i=1}^{k}\gamma_{i}\right)  ^{2\ell+1}%
}{(2\ell+1)!}.
\end{align*}

\end{proof}

\section{Some facts about Hankel and Bessel functions}

In this section, we state some properties of spherical Hankel and Bessel functions.

\begin{lemma}
\label{lem:hankelpiecewise} It holds
\[
h_{0}^{(1)}(x)=\frac{1}{x}\left(  \sin x+\operatorname*{i}\cos x\right)
\quad\text{and\quad}j_{0}(x)=\frac{\sin x}{x}%
\]
In particular it holds for $x\neq0$%
\begin{equation}
x\left\vert h_{0}(x)\right\vert =1\quad\text{and\quad}\left\vert xj_{0}\left(
x\right)  \right\vert \leq\left\vert \sin x\right\vert \leq\frac{2\left\vert
x\right\vert }{1+\left\vert x\right\vert } \label{modh0j0}%
\end{equation}
and
\begin{equation}
\left\vert \left(  h_{0}^{(1)}\right)  ^{\prime}\left(  x\right)  \right\vert
^{2}=\frac{1}{x^{2}}+\frac{1}{x^{4}}\quad\text{and\quad}\left\vert
j_{0}^{\prime}(x)\right\vert ^{2}=\left(  \frac{\cos x}{x}-\frac{\sin x}%
{x^{2}}\right)  ^{2} \label{modh0j02}%
\end{equation}

\end{lemma}

\begin{lemma}
\label{lem:integralhankelbessel}For $d=3$ and $m=0$ the fundamental system
satisfies the integral estimates%
\begin{align*}
\left\Vert h_{0}^{(1)}\left(  \frac{\omega}{c_{j}}\cdot\right)  \mathfrak{g}%
_{0}\right\Vert _{L^{2}\left(  \tau_{j}\right)  }^{2}  &  \leq\left(
\frac{c_{j}}{\omega}\right)  ^{2}h_{j},\\
\left\Vert \left(  h_{0}^{(1)}\right)  ^{\prime}\left(  \frac{\omega}{c_{j}%
}\cdot\right)  \mathfrak{g}_{0}\right\Vert _{L^{2}\left(  \tau_{j}\right)
}^{2}  &  \leq\left(  1+\left(  \frac{c_{j}}{z_{j-1}}\right)  ^{2}\right)
\left(  \frac{c_{j}}{\omega}\right)  ^{2}h_{j},\\
\left\Vert j_{0}\left(  \frac{\omega}{c_{j}}\cdot\right)  \mathfrak{g}%
_{0}\right\Vert _{L^{2}\left(  \tau_{j}\right)  }^{2}  &  \leq\left(
\frac{2z_{j}}{c_{j}+z_{j}}\right)  ^{2}\left(  \frac{c_{j}}{\omega}\right)
^{2}h_{j},\\
\left\Vert j_{0}^{\prime}\left(  \frac{\omega}{c_{j}}\cdot\right)
\mathfrak{g}_{0}\right\Vert _{L^{2}\left(  \tau_{j}\right)  }^{2}  &
\leq\frac{16z_{j}^{4}}{\left(  2c_{j}^{2}+z_{j}^{2}\right)  ^{2}}\left(
\frac{c_{j}}{\omega}\right)  ^{2}h_{j},
\end{align*}
as well as the pointwise estimates for $r\in\tau_{j}$ (with $y=\omega r/c_{j}%
$)%
\[%
\begin{array}
[c]{lll}%
\left\vert h_{0}^{(1)}\left(  y\right)  \right\vert \leq\dfrac{1}{y}, & \quad
& \left\vert \left(  h_{0}^{(1)}\right)  ^{\prime}\left(  y\right)
\right\vert \leq\dfrac{1}{y}\left(  1+\dfrac{1}{y}\right)  ,\\
\left\vert j_{0}\left(  y\right)  \right\vert \leq\dfrac{2}{1+y}, & \quad &
\left\vert j_{0}^{\prime}\left(  y\right)  \right\vert \leq\dfrac{4y}{2+y^{2}%
}.
\end{array}
\]

\end{lemma}

\begin{proof}
By using Lemma \ref{lem:hankelpiecewise}, we have%
\begin{align*}
\left\Vert h_{0}^{(1)}\left(  \frac{\omega}{c_{j}}\cdot\right)  \mathfrak{g}%
_{0}\right\Vert _{L^{2}\left(  \tau_{j}\right)  }^{2}  &  =\int_{x_{j-1}%
}^{x_{j}}\left(  \frac{\omega}{c_{j}}r\right)  ^{-2}r^{2}dr=\left(
\frac{c_{j}}{\omega}\right)  ^{2}h_{j},\\
\left\Vert \left(  h_{0}^{(1)}\right)  ^{\prime}\left(  \frac{\omega}{c_{j}%
}\cdot\right)  \mathfrak{g}_{0}\right\Vert _{L^{2}\left(  \tau_{j}\right)
}^{2}  &  =\int_{x_{j-1}}^{x_{j}}\left\vert h_{0}^{\prime}\left(  \frac
{\omega}{c_{j}}r\right)  \right\vert ^{2}r^{2}dr\leq\left(  \frac{c_{j}%
}{\omega}\right)  ^{2}\left(  1+\left(  \frac{c_{j}}{z_{j-1}}\right)
^{2}\right)  h_{j}.
\end{align*}
For the spherical Bessel functions we get
\begin{align*}
\left\Vert j_{0}\left(  \frac{\omega}{c_{j}}\cdot\right)  \mathfrak{g}%
_{0}\right\Vert _{L^{2}\left(  \tau_{j}\right)  }^{2}  &  =\int_{x_{j-1}%
}^{x_{j}}\left(  \frac{\sin\left(  \frac{\omega}{c_{j}}r\right)  }{\left(
\frac{\omega}{c_{j}}r\right)  }\right)  ^{2}r^{2}dr=\left(  \frac{c_{j}%
}{\omega}\right)  ^{3}\int_{z_{j-1}/c_{j}}^{z_{j}/c_{j}}\left(  \sin y\right)
^{2}dy\\
&  \leq4\left(  \frac{c_{j}}{\omega}\right)  ^{3}\int_{z_{j-1}/c_{j}}%
^{z_{j}/c_{j}}\frac{y^{2}}{\left(  1+y\right)  ^{2}}dy\leq\left(  2\frac
{c_{j}}{\omega}\frac{z_{j}}{c_{j}+z_{j}}\right)  ^{2}h_{j}\\
\left\Vert j_{0}^{\prime}\left(  \frac{\omega}{c_{j}}\cdot\right)
\mathfrak{g}_{0}\right\Vert _{L^{2}\left(  \tau_{j}\right)  }^{2}  &  =\left(
\frac{c_{j}}{\omega}\right)  ^{3}\int_{z_{j-1}/c_{j}}^{z_{j}/c_{j}}\left(
\cos x-\frac{\sin x}{x}\right)  ^{2}dx.
\end{align*}
We set $\kappa\left(  x\right)  :=\cos x-\frac{\sin x}{x}$ and a
straightforward calculation leads to $\left\vert \kappa^{\prime\prime}\left(
x\right)  \right\vert \leq2$. This together with the trivial estimate $\left\vert \kappa\left(
x\right)  \right\vert \leq2$ leads to%
\[
\left\vert \kappa\left(  x\right)  \right\vert \leq\min\left\{  x^{2}%
,2\right\}  \leq\frac{4x^{2}}{x^{2}+2}.
\]
Hence,%
\[
\left\Vert j_{0}^{\prime}\left(  \frac{\omega}{c_{j}}\cdot\right)
\mathfrak{g}_{0}\right\Vert _{L^{2}\left(  \tau_{j}\right)  }^{2}\leq
\frac{16z_{j}^{4}}{\left(  2c_{j}^{2}+z_{j}^{2}\right)  ^{2}}\left(
\frac{c_{j}}{\omega}\right)  ^{2}h_{j}.
\]
The first three pointwise estimates follow from \eqref{modh0j0} and
\eqref{modh0j02}. For the last one, we estimate%
\[
\left\vert j_{0}^{\prime}\left(  y\right)  \right\vert =\frac{1}{y}\left\vert
\kappa\left(  y\right)  \right\vert \leq\frac{4y}{2+y^{2}}.
\]%
\end{proof}

\bibliographystyle{abbrv}
\bibliography{multiDHelmholtz}

\end{document}